\documentclass[11pt,a4paper,reqno]{amsart}
\usepackage{hyperref}
\usepackage[all]{xy}

\font\cyr=wncyr10 scaled \magstep1%
\def\Sh{\text{\cyr Sh}}

\input amssymb.sty

\thanks{Bitan was partially supported by the Hermann Minkowski Center for Geometry.}

\newcommand{\surj}{\twoheadrightarrow}

\DeclareMathOperator{\cok}{coker}
\DeclareMathOperator{\im}{im}

\newcommand{\af}{\text{af}}
\newcommand{\et}{{\text{\rm \'et}}}
\newcommand{\red}{\text{red}}
\newcommand{\sh}{\mathrm{sh}}

\newcommand{\Pic}{\mathrm{Pic~}}

\newcommand{\iy}{\infty}
\newcommand{\wt}{\widetilde}
\newcommand{\wh}{\widehat}
\newcommand{\ov}{\overline}
\newcommand{\un}{\underline}

\newcommand{\BG}{\mathbb{G}}
\newcommand{\BF}{\mathbb{F}}

\newcommand{\BQ}{\mathbb{Q}}

\newcommand{\BZ}{\mathbb{Z}}
\newcommand{\A}{\mathbb{A}}

\newcommand{\fp}{\mathfrak{p}}

\newcommand{\fg}{\mathfrak{g}}
\newcommand{\fX}{\mathfrak{X}}
\newcommand{\Oiy}{\CO_{\{\iy\}}}
\newcommand{\Hom}{\text{Hom}}

\newcommand{\Sp}{\text{Spec} \,}
\newcommand{\Lie}{\text{Lie}}

\newcommand{\ur}{{\mathrm{un}}}
\newcommand{\Ext}{\text{Ext}}
\renewcommand{\sc}{{\mathrm{sc}}}

\newcommand{\bk}{\bigskip}
\newcommand{\fc}{\frac}
\newcommand{\G}{\Gamma}

\newcommand{\s}{\sigma}

\newcommand{\D}{\mathcal{D}}

\newcommand{\CG}{\mathcal{G}}
\newcommand{\CGL}{\mathcal{GL}}

\newcommand{\MA}{\mathcal{A}}

\newcommand{\T}{\mathcal{T}}
\newcommand{\F}{\mathcal{F}}
\newcommand{\CL}{\mathcal{L}}
\newcommand{\lT}{\un{T}^{\text{lft}}}

\newcommand{\CO}{\mathcal{O}}
\newcommand{\B}{\mathcal{B}}
\renewcommand{\a}{\alpha}
\newcommand{\dl}{\delta}

\newcommand{\lm}{\lambda}
\newcommand{\Lm}{\Lambda}
\newcommand{\om}{\omega}

\newcommand{\Z}{\mathbb{Z}}

\newcommand{\la}{\langle}
\newcommand{\ra}{\rangle}

\newtheorem{theorem}{Theorem}[section]
\newtheorem*{maintheorem*}{Main Theorem}
\newtheorem{prop}[theorem]{Proposition}
\newtheorem{cor}[theorem]{Corollary}
\newtheorem{lem}[theorem]{Lemma}

\theoremstyle{definition}
\newtheorem{definition}[theorem]{Definition}
\newtheorem{remark}[theorem]{Remark}
\newtheorem{example}[theorem]{Example}

\begin{document}

\date{}

\baselineskip 20pt
\setcounter{equation}{0}
\pagestyle{plain}
\pagenumbering{arabic}

\title{A building-theoretic approach to relative Tamagawa numbers of semisimple groups over global function fields}

\author{Rony A. Bitan}

\author{Ralf K\"ohl (n\'e Gramlich)}

\begin{abstract}
Let $G$ be a semisimple group defined over a global function field $K$ of a rational curve, 
not anistropic of type $A_n$. 
We express the (relative) Tamagawa number of $G$ in terms of local data 
including the number $t_\iy(G)$ of types in one orbit of a special vertex 
in the Bruhat--Tits building of $G_\iy(\hat{K}_\iy)$ for some place $\iy$ 
and the class number $h_\iy(G)$ of $G$ at $\iy$.
\end{abstract}

\maketitle

\pagestyle{headings}

\markright{RELATIVE TAMAGAWA NUMBERS}

\section{Introduction} \label{section1}
Let $C$ be a smooth, projective and irreducible algebraic curve 
defined over the finite field $\BF_q$ and let $K = \BF_q(C)$ be its function field. 
Let $G$ be a (connected) semisimple  group defined over $K$. 
The Tamagawa number $\tau(G)$ of $G$ is defined as the covolume of the group $G(K)$ of $K$-rational points in the adelic group $G(\mathbb{A})$ (embedded diagonally as a discrete subgroup) with respect to the volume induced by Tamagawa measure on $G(\mathbb{A})$ 
(see \cite{Weil}, \cite{Clozel} and Section~\ref{section4} below).

Let $\pi:G^\sc \to G$ be the universal covering and let $F = \ker(\pi)$ be the fundamental group. 
We assume that 
$G$ is not an almost direct product of anisotropic almost simple groups of type $A_n$ and that 
$(\text{char}(K),|F|)=1$. 
According to the Weil conjecture $\tau(G^\sc)=1$. 
By \cite{Harder:1974} the Weil conjecture is known to hold for split $G$. 
A geometric proof of Weil's conjecture by Gaitsgory--Lurie has been announced in \cite[1.2.3]{Gaitsgory:2011}, see also: \cite{Lurie:homepage}.

In the present article we investigate the relative Tamagawa number 
$\fc{\tau(G)}{\tau(G^\sc)}$ from a building-theoretic point of view --  
in the situation that $G$ is locally {\em isotropic} everywhere. 
Let $\iy$ be some closed point of $K$ and 
let $\A_\iy := \hat{K}_\iy \times \prod_{\fp \neq \iy} \hat{\CO}_\fp$ be the subring of $\{\iy\}$-integral ad\`eles 
in the ad\`ele ring $\A$. 
The double cosets set $\text{Cl}_\iy(G) := G(\A_\iy) \backslash G(\A) / G(K)$ is finite (\ref{class number of semisimple groups}). 
This allows us to split off the class number $h_\iy(G) = |\text{Cl}_\iy(G)|$ 
and proceed by computing the co-volume of $G(K)$ in the trivial coset $G(\A_\iy)G(K)$ 
w.r.t. the Tamagawa measure by considering the natural action of $G(\A_\iy)$ on the Bruhat--Tits building of $G_\iy(\hat{K}_\iy)$, 
resulting in formula (\ref{Tamagawa number semisimple formula}) below
\begin{eqnarray*}
\tau(G) &  = & q^{-(g-1)\dim(G)}\cdot h_\iy(G) \cdot i_\iy(G) \cdot \prod_{\fp} \om_\fp(\un{G}^0_\fp(\hat{\CO}_\fp))
\end{eqnarray*}
where for each $\fp$, $\om_\fp$ is some multiplicative local Haar-measure, 
$\un{G}^0_\fp$ stands for the connected component of the Bruhat--Tits $\hat{\CO}_\fp$-model at some special point 
and $i_\iy(G)$ is an arithmetic invariant related to $G_\iy := G \otimes_K \hat{K}_\iy$.   
The key obstruction for using this formula is to determine a fundamental domain 
for the action of $\CG^0(\Oiy)$ on a $G_\iy(\hat{K}_\iy)$-orbit of the Bruhat--Tits building of $G_\iy(\hat{K}_\iy)$ 
where $\CG^0$ is a flat connected smooth and finite type model of $G$ 
defined over the ring $\Oiy$ of $\{\iy\}$-integers in $K$ w.r.t. the one of $\CG^\sc(\Oiy)$ associated to $G^\sc$.     

In Proposition~\ref{covolume equality} below we will see 
that for computing the {\em relative} local volumes $\fc{i_\iy(G)}{i_\iy(G^\sc)}$ 
it suffices to compare orbits under $G^\sc_\iy(\hat{K}_\iy)$ and $G_\iy(\hat{K}_\iy)$, 
whose behavior is controlled by the number $t_\iy(G)$ of types in the $G_\iy(\hat{K}_\iy)$-orbit of the fundamental special vertex 
and by the number $j_\iy(G)$ expressing the comparison between the discrete subgroups $\CG^\sc(\Oiy)$ 
and $\CG^0(\Oiy)$.

Altogether we then arrive at the following main result of our article. 
By $K^s_\iy$ we denote the separable closure of $\hat{K}_\iy$ 
with Galois group $\fg_\iy = \mathrm{Gal}(K^s_\iy/\hat{K}_\iy)$ 
and inertia subgroup $I_\iy = \mathrm{Gal}(K^s_\iy/K^\ur_\iy)$. 
Moreover, $\s_\iy$ denotes a generator of $\fg_\iy/I_\iy$, i.e., 
the map $\s_\iy : x \mapsto x^{|k_\iy|}$ where $k_\iy$ is the residue field of $\hat{K}_\iy$. 
Let $F_\iy:=\ker[G^\sc_\iy \to G_\iy]$ and $\wh{F_\iy} := \Hom(F_\iy \otimes \hat{K}_\iy^s,\BG_{m,\hat{K}_\iy^s})$.

\begin{maintheorem*} \label{maintheorem}
With these notations and assuming the Weil conjecture validity one has
$$ \tau(G) =  h_\iy(G) \cdot \fc{t_\iy(G)}{j_\iy(G)}. $$  
The number $t_\iy(G)$ satisfies 
$$ t_\iy(G) = |H^1(I_\iy,F_\iy(\hat{K}_\iy^s))^{\s_\iy}| = |\wh{F_\iy}^{\fg_\iy}|$$
and: 
$$ j_\iy(G) = \fc{|H^1_\et(\Oiy,\F)|}{|\F(\Oiy)|} $$
where $\F:= \ker[\CG^\sc \to \CG^0]$.   \\
If in particular $G$ is quasi-split and the genus $g$ of the curve $C$ is $0$ then $j_\iy(G)=1$ and so 
$$ \tau(G) =  h_\iy(G) \cdot t_\iy(G) = h_\iy(G) \cdot |\wh{F_\iy}^{\fg_\iy}|. $$  
\end{maintheorem*}

\begin{cor} (Cor.~\ref{G split} below) 
If $G$ is split and $g=0$ then $h_\iy(G)=1$ and $\tau(G) = t_\iy(G) = |F|$. 
\end{cor}

\begin{cor} (Cor.~\ref{G adjoint} below)
If $G$ is adjoint (not necessarily split) and $g=0$ then $h_\iy(G)=1$ and $\tau(G) = t_\iy(G) = |\wh{F}^\fg|$, 
where $\wh{F}:=\mathrm{Hom}(F(K^s),\BG_{m,K^s})$ and $\fg := \mathrm{Gal}(K^s/K)$.   
\end{cor}

Our method of proof is a combination of geometric group theory and
cohomology. Our approach is independent of Prasad's covolume formula described
in \cite{Pr2}, but it is likely that with some effort it can be used to deduce
our Main Theorem. 

As an application in case the group $G$ is quasi-split and the genus $g$ of $C$ is $0$,  
we combine our result with \cite[Formula (3.9.1')]{Ono1} and
the techniques from \cite[\S~8.2]{PR} in order to relate the cokernels of
Bourqui's degree maps $\mathrm{deg}_{T^\sc}$ and $\mathrm{deg}_T$ from \cite[Section~2.2]{Bou}, where
$T^\sc$ and $T$ denote suitable Cartan subgroups of $G^\sc$ and $G$ respectively;
cf.\ Proposition~\ref{D=1}. These concrete computations allow us to also provide a wealth of examples in Section~\ref{section6} for which we compute the relative Tamagawa numbers. 
We also demonstrate the result in a case of a split group defined over the function field of an elliptic curve (Remark \ref{PGLn over elliptic curve}).   

This article is organized as follows: 
In the preliminary Section~\ref{section2} we fix the relevant notions from Bruhat--Tits theory. 
In Section~\ref{section3} we compute volumes of parahoric subgroups over local fields, their maximal unramified extensions, and their valuation rings. 
In Section~\ref{section4} we revise the definition of the Tamagawa number of semisimple $K$-groups  
and establish a decomposition of $G(\A)/G(K)$ enabling us to express $\tau(G)$ in terms of a global invariant and a local one.  
In Section~\ref{section5} we compute cohomology groups over rings of $S$-integers with $|S|=1$,   
use Bruhat--Tits theory and Serre's formula (\cite[p.~84]{Ser1}, \cite[Corollary~1.6]{Bass/Lubotzky:2001}) 
in order to derive the above-mentioned formula (\ref{Tamagawa number semisimple formula}) for computing the Tamagawa number. 
In Section~\ref{section6} we express the number $t_\iy(G)$ of types in the orbit of a special point 
in terms of $F_\iy$, accomplishing the proof of our Main Theorem.  
The final Section~\ref{section7} addresses the above-mentioned application and examples.

\bigskip \noindent
{\bf Acknowledgements:} 
The authors thank M.~Borovoi, D.~Bourqui, B.~Conrad, P.~Gille, C.~D.~Gonz\'alez-Avil\'es, B.~Kunyavski\u\i, 
Q.~Liu, G.~McNinch, G.~Prasad and R.~Weissand for valuable discussions concerning the topics of the present article. 
They also thank A.~Rapinchuk and an anonymous referee for helpful comments on an earlier version of this article.  

\bk

\section{Basic notions from Bruhat-Tits theory} \label{Local invariants}\label{section2}
We retain the notation from Section~\ref{section1},  
only we assume $G$ to be quasi-split. 
In Section \ref{section3}, however, this assumption will be dropped. 
Since $K$ is a function field the valuations defined on $K$ are non-Archimedean. 
For any prime $\fp$ of $K$, let $K_\fp$ be the localization of $K$ at $\fp$, let $\hat{K}_\fp$ be its completion there  
and let $\CO_\fp$ and $\hat{\CO}_\fp$ be their ring of integers respectively. 
Let $k_\fp = \hat{\CO}_\fp / \fp$ be the corresponding (finite) residue field. 
Then $G_\fp = G \otimes_K \hat{K}_\fp$ is semisimple.  
The second assumption $(\text{char}(K),|F|)=1$ says that $\pi$ is separable.   
$T_\fp = T \otimes_K \hat{K}_\fp$ is a Cartan subgroup in $G_\fp$.  
Let $(X^*(T_\fp),\Phi,X_*(T_\fp),\Phi^\vee)$ be the root datum of $(G_\fp,T_\fp)$ and let $W$ be the associated constant Weyl group. 
Let $\B_\fp$ be the Bruhat--Tits building associated to the adjoint group of $G_\fp$ 
(cf.\ \cite[Section~7]{Bruhat/Tits:1972}, also \cite[Chapter~11]{Abramenko/Brown:2008}) 
and let $\MA$ be the apartment in $\B_\fp$ corresponding to $T_\fp$. 

\bk

We fix a special vertex $x \in \MA$, 
i.e., a vertex whose isotropy group in the setwise stabilizer of $\MA$ is isomorphic to $W$.   
Since the Bruhat--Tits building $\B_\fp$ is locally finite, the stabilizer
$P_x$ of $x$ in $G_\fp(\hat{K}_\fp)$ 
is a compact subgroup of $G_\fp(\hat{K}_\fp)$. 
Let $\un{G}_x$ be the Bruhat--Tits model associated to $P_x$, 
i.e., such that $\un{G}_x(\hat{\CO}_\fp) = P_x$. 
Denote by $\ov{G}_x$ the reduction modulo $\fp$ of $\un{G}_x$ 
and by $\un{G}_x^0$ the open subscheme of $\un{G}_x$ 
whose reduction is the identity component $\ov{G}_x^0$ of $\ov{G}_x$. 
Let $\un{T}_\fp$ be the N\'eron--Raynaud $\hat{\CO}_\fp$-model 
(shortly referred as NR-model) of $T_\fp$ which is of finite type, 
i.e., such that $\un{T}_\fp(\hat{\CO}_\fp)$ is the maximal compact subgroup of $T_\fp(\hat{K}_\fp)$ 
(see Theorem 2 in \cite[\S~10.2]{BLR} 
and \cite[\S~3.2]{CY} for an explicit construction).  
Denote by $\un{T}_\fp^0$ its subscheme having a connected special fiber.   
$\un{T}_\fp^0(\hat{\CO}_\fp)$ is the pointwise stabilizer of $\MA$ 
and is a subgroup of $\un{G}_x^0(\hat{\CO}_\fp)$. 
Since $G_\fp$ is semisimple and the residue field $k_\fp$ is finite, 
the adjoint group of $G_\fp(\hat{K}_\fp)$ 
permutes transitively the special vertices (see \cite[\S~2.5]{Tits}). 
If $\Phi$ is not reduced, we adapt the convention of Prasad in \cite[\S~1.2]{Pr2} and Gross in \cite[\S~4]{Gro} 
which is the following: for each component of the local Dynkin diagram of the type
\begin{equation*} 
\bullet \Longleftarrow \bullet \ \text{---} \ \bullet \cdot \cdot \cdot 
\bullet \ \text{---} \ \bullet \Longleftarrow \bullet  
\end{equation*}
we choose the special vertex at the right end of the diagram. 
Now $\un{G}_x$ is well defined up to isomorphism
and would be denoted from now on by $\un{G}_\fp$. 
$x$ is called the {\em fundamental special vertex} of $\B_\fp$.

\begin{remark} \label{quasi-trivial Cartan}
If $G$ is either simply connected or adjoint, 
it is a finite product of restriction of scalars $R_{K_i/K}(G_i)$ 
where each $K_i$ is a separable extension of $K$ and $G_i$ is split and simple.  
If $G$ is also quasi-split, its Cartan subgroup $T$ 
being the product of the centralizers of the split tori of the $G_i$'s, 
is a maximal torus (see in the proof of \cite[Prop.~16.2.2.]{Spr}) 
being quasi-trivial, i.e. a finite product of Weil's tori $R_{K_i/K}(\BG_m)$.   
In this case $\un{T}_\fp$ is connected at any $\fp$ (see \cite[Prop.~2.4]{NX}). 
\end{remark}

\bk

\section{Volumes of parahoric subgroups} \label{section3}
We retain our restriction of $G$ to be an almost direct product of anisotropic almost simple groups of type $A_n$, 
implying that $G_\fp$ is $K_\fp$-isotropic (see \cite[4.3~and~4.4]{BT3}).   
As $\hat{K}_\fp$ is locally compact 
its underlying additive group admits a Haar measure which is unique up to a scalar multiple. 
As in \cite[\S~2.1]{Weil} we normalize such a measure $dx_\fp$ by $dx_\fp(\hat{\CO}_\fp) = 1$. 
This induces a multiplicative Haar measure $\varpi_\fp$ 
on the locally compact group $G_\fp(\hat{K}_\fp)$, see \cite[\S~2.2]{Weil}. 
Our choice of the Bruhat--Tits model in the preceding section allows us to
easily compute the volume of the fundamental parahoric subgroup with respect
to this Haar measure.

\begin{prop} \label{volume of identity component}  
Up to a multiplication by a scalar of $K_\fp^\times$, 
which is uniquely determined by the normalization of a multiplicative Haar measure on $G_\fp(\hat{K}_\fp)$,  
the volume of $\un{G}_\fp^0(\hat{\CO}_\fp)$ w.r.t. this measure is  
$|\ov{G}_\fp^0(k_\fp)| \cdot |k_\fp|^{-\dim G_\fp}.$ 
\end{prop}

\begin{proof}
Assume first that $G_\fp$ is quasi-split over $\hat{K}_\fp$. 
As $\un{G}_\fp^0$ is smooth and $\hat{\CO}_\fp$ is Henselian and therefore complete,  
the reduction of its group of $\hat{\CO}_\fp$-points is surjective (see \cite[Proposition~2.3/5]{BLR}). 
Consider the exact sequence
$$ 
1 \to \un{G}_\fp^1(\hat{\CO}_\fp) \to \un{G}_\fp^0(\hat{\CO}_\fp) \stackrel{\red}{\longrightarrow} \ov{G}_\fp^0(k_\fp) \to 1. 
$$
$\un{G}_\fp^1(\hat{\CO}_\fp)$ is the reduction preimage of $1_d$ in $\ov{G}_\fp^0(k_\fp)$,   
where $d = \dim \un{G}_\fp^0 = \dim G_\fp$.  
Since $T_\fp$ is maximal and $G_\fp$ is quasi-split, by \cite[Corollary~4.6.7]{BT2} 
$\un{G}_\fp^0(\hat{\CO}_\fp) = \un{T}_\fp^0(\hat{\CO}_\fp) \fX(\hat{\CO}_\fp)$ 
where $\fX(\hat{\CO}_\fp)$ is the group generated by the root subgroups 
each fixing an half apartment containing $x$. 
The preimage of $1_d$ in $\un{G}_\fp^0(\hat{\CO}_\fp) / \un{T}_\fp^0(\hat{\CO}_\fp)$ 
is homeomorphic to the additive group $\fp^{|\Phi|}$. 
The preimage of $1_d$ in $\un{T}_\fp^0(\hat{\CO}_\fp)$ is isomorphic to $(1+\fp)^{\dim T_\fp}$, 
being homeomorphic to $\fp^{\dim T_\fp}$. 
Together, $\Lie(\un{G}_\fp^1(\hat{\CO}_\fp)) \cong \fp^{\dim T_\fp + |\Phi|} = \fp^d$. 
The normalization condition $dx_\fp(\hat{\CO}_\fp)=1$ is equivalent to $dx_\fp(\fp) = |k_\fp|^{-1}$ 
implying that
$$ \varpi_\fp(\un{G}_\fp^1(\hat{\CO}_\fp)) = \bigwedge_{i=1}^{d} dx_\fp(\fp^d) = |k_\fp|^{-d}. $$
Now from the exactness of the sequence we deduce
\begin{equation*}
\varpi_\fp(\un{G}_\fp^0(\hat{\CO}_\fp)) = [\un{G}_\fp^0(\hat{\CO}_\fp):\un{G}_\fp^1(\hat{\CO}_\fp)] \cdot \varpi_\fp(\un{G}_\fp^1(\hat{\CO}_\fp))  
                               = |\ov{G}_\fp^0(k_\fp)| \cdot |k_\fp|^{-d}.  
\end{equation*}
If $G_\fp$ is not quasi-split this computation needs to be applied to an inner form of $G_\fp$ 
that is quasi-split, 
w.r.t. the twisted measure of $\varpi_\fp$, 
being again some scalar multiple of $\varpi_\fp$ (see \cite[Prop.~4.7]{Gro}).  
\end{proof}

\begin{remark} \label{G = G0 in unramified extension} \cite[4.6.22]{BT2}
If $G_\fp$ splits over an unramified extension, then $\un{G}_\fp(\hat{\CO}_\fp)=\un{G}_\fp^0(\hat{\CO}_\fp)$. 
\end{remark}

Let $\pi_\fp: G_\fp^\sc \to G_\fp$ be the universal covering of $G_\fp$. 
According to \cite[4.4.18(VI)]{BT2}, 
the cover $\pi_\fp$ restricted to $T^\sc_\fp$ extends to a homomorphism
$\un{T^\sc}_\fp \to \un{T}_\fp$ over $\Sp \hat{\CO}_\fp$. 
Together with the associated root subgroups $\hat{\CO}_\fp$-scheme $\fX$, which is equal for both $G^\sc_\fp$ and $G_\fp$, 
this homomorphism over $\Sp \hat{\CO}_\fp$ extends to a homomorphism $\un{G^\sc}_\fp
\to \un{G}_\fp$ of the Bruhat--Tits schemes. 
Let $\un{F}_\fp := \ker[\un{G^\sc}_\fp \to \un{G}_\fp]$. 
It is finite, flat and its generic fiber is $F_\fp$. 
As $F_\fp$ is central, so is $\un{F}_\fp$, 
thus embedded in $\un{T^\sc}_\fp$. 

Let $\hat{K}_\fp^\ur$ be the maximal unramified extension of $\hat{K}_\fp$, 
i.e., the strict henselization of $\hat{K}_\fp$ with ring of integers $\hat{\CO}_\fp^\sh$ and algebraically closed residue field $k_\fp^s$. 
Let $\hat{K}_\fp^s$ be a separable closure of $\hat{K}_\fp$ containing $\hat{K}_\fp^\ur$ 
and let $I_\fp = \mathrm{Gal}(\hat{K}_\fp^s/\hat{K}_\fp^\ur)$ be the inertia subgroup of $\fg_\fp = \mathrm{Gal}(\hat{K}_\fp^s/\hat{K}_\fp)$. 
Let $\s_\fp$ be a generator of $\fg_\fp/I_\fp$, 
i.e., the map $\sigma_\fp : x \mapsto x^{|k_\fp|}$ where as above $k_\fp$ is the residue field of $\hat{K}_\fp$.

\begin{prop} \label{tori surjection} 
Any separable isogeny $\pi_\fp: T_\fp \to T'_\fp$ of $\hat{K}_\fp$-tori 
can be extended to an isogeny $\un{\pi}:\un{T^\ur}_\fp \to \un{(T')^\ur}_\fp$ over $\hat{\CO}_\fp^\sh$,  
inducing a surjection $ \un{T^\ur}^0_\fp(\hat{\CO}_\fp^\sh) \to \un{(T')^\ur}^0_\fp(\hat{\CO}_\fp^\sh)$.  
\end{prop}

\begin{proof}
Any $\hat{K}_\fp^\ur$-torus $T_\fp$ admits a decomposition, i.e., an exact sequence of $\hat{K}_\fp^\ur$-tori
\begin{equation} \label{decomposition}
1 \to T_{I,\fp} \to T_\fp \to T_{a,\fp} \to 1
\end{equation}
on which $T_{I,\fp}$ is the maximal subtorus of $T_\fp$ splitting over $\hat{K}_\fp^\ur$ 
and $T_{a,\fp}$ is $I_\fp$-anisotropic, i.e., such that
$X^*(T_{a,\fp})^{I_\fp}=\{0\}$.

We denote by $\lT_\fp$ the locally of finite type (lft) NR-model  
of $T_\fp$ defined over $\Sp \hat{\CO}_\fp^\sh$ (see \ref{section2}).  
Let $j_*$ be the functor taking algebraic $\hat{K}_\fp^\ur$-tori to their lft-N\'eron models.  
Since $T_{I,\fp}^\ur$ is $\hat{K}_\fp^\ur$-split, we have $R^1 j_* = 0$ 
(cf.\ the beginning of the proof of III.C.10 in \cite{Mil2}).  
Thus the exact sequence (\ref{decomposition}) can be extended to
\begin{equation} \label{NR models decomposition}
1 \to \lT_{I,\fp} \to \lT_\fp  \to \lT_{a,\fp} \to 1.
\end{equation}
According to \cite[Proposition~4.2(b)]{LL} 
the groups of $k_\fp^s$-points of the connected components of the reductions of these models
fit into the exact sequence 
$$ 1 \to \ov{T_{I,\fp}}^0(k_\fp^s) \to \ov{T}_\fp^0(k_\fp^s) \to \ov{T_{a,\fp}}^0(k_\fp^s) \to 1. $$
As $k_\fp^s$ is algebraically closed, this sequence implies the corresponding exact sequence of $k_\fp^s$-schemes
$$ 1 \to \ov{T_{I,\fp}}^0           \to \ov{T}_\fp^0         \to \ov{T_{a,\fp}}^0          \to 1. $$
Notice that the identity components of the lft NR-models coincide with the ones of the finite type models. 
Thus the reduction preimages of the latter $k_\fp^s$-schemes, 
embedded in the $\hat{\CO}_\fp^\sh$-schemes in sequence (\ref{NR models decomposition}), 
yield the exact sequence of the identity components over $\hat{\CO}_\fp^\sh$
\begin{equation} \label{NR models identity components decomposition}
1 \to \un{T_I}_\fp^0 \to \un{T}_\fp^0  \to \un{T_a}_\fp^0 \to 1.
\end{equation}
Now let $\pi_\fp:T_\fp \to T'_\fp$ be an isogeny of $\hat{K}_\fp$-tori.  
Denote by $T_\fp^\ur$ and $(T'_\fp)^\ur$ these tori tensored with $\hat{K}_\fp^\ur$. 
Then applying the decomposition (\ref{NR models identity components decomposition}) on both $T_\fp^\ur$ and $(T'_\fp)^\ur$ 
results in the exact sequences
\begin{align} \label{isogeny sequences decompositions} 
1 \to \un{T^\ur_I}_\fp^0       & \to \un{T^\ur}_\fp^0      \to \un{T^\ur_a}_\fp^0    \to 1, \\
1 \to \un{(T'_I)^\ur}_\fp^0  & \to \un{(T')^\ur}_\fp^0   \to \un{(T'_a)^\ur}_\fp^0   \to 1. \nonumber
\end{align}
If we show that the left-hand and right-hand groups in the upper sequence surject onto the corresponding groups in the lower one, 
then surjection of the middle groups will follow. 
On the left hand side $T_{I,\fp}^\ur$ and $(T'_{I,\fp})^\ur$ are isogenous and $\hat{K}_\fp^\ur$-split. 
Then $\pi_I := \ker[T_I^\ur \surj (T'_I)^\ur]$ is a finite $\hat{K}_\fp^\ur$-split group of multiplicative type. 
Thus the Kummer exact sequence of $\hat{K}_\fp^\ur$-schemes
$$ 1 \to \pi_I \to T_{I,\fp}^\ur \to (T'_{I,\fp})^\ur \to 1 $$ 
extends to the exact sequence of corresponding schemes over $\hat{\CO}_\fp^\sh$
$$ 1 \to \un{\pi_I} \to \un{T_I^\ur}_\fp \to \un{(T'_I)^\ur}_\fp \to 1 $$
showing the desired surjection on the left-hand side (notice that both $\un{T_I^\ur}_\fp$ and $\un{(T'_I)^\ur}_\fp$ split 
over $\hat{\CO}_\fp^\sh$ thus connected, i.e. coincide with their identity component  
(see Remark \ref{G = G0 in unramified extension}). 

Both groups $T_{a,\fp}^\ur$ and $(T'_{a,\fp})^\ur$ on the right-hand side 
of sequences (\ref{isogeny sequences decompositions}) are $I_\fp$-anisotropic. 
Therefore their NR-models coincide with the finite type (classical) N\'eron model. 
In that case, according to \cite[Section~7.3, Proposition~6]{BLR}, the $\hat{K}_\fp^\ur$-isogeny $T_{a,\fp}^\ur \to (T'_{a,\fp})^\ur$ 
extends to a $\hat{\CO}_\fp^\sh$-isogeny $\un{T_a^\ur}_\fp \to \un{(T'_a)^\ur}_\fp$, 
such that the surjection holds for the identity components, see Definition 4
of {\em loc.\ cit.}
Hence we deduce the surjection $ \un{T^\ur}_\fp^0 \surj \un{(T')^\ur}_\fp^0$.

Further, as the degree of the latter $\hat{\CO}_\fp^\sh$-isogeny is prime to $\text{char}(\hat{K}_\fp)$ (see \ref{section1}),  
its kernel $\un{F^\ur}_\fp$ has a smooth reduction as well. 
Thus the exact sequence of the reduction groups over the algebraically closed residue field $k_\fp^s$
$$ 1 \to \ov{F^\ur}_\fp(k_\fp^s) \to \ov{T^\ur}^0_\fp(k_\fp^s) \to \ov{(T')^\ur}_\fp^0(k_\fp^s) \to 1$$ 
implies the exactness of the reduction preimage groups of $\hat{\CO}_\fp^\sh$-points
\begin{equation*} 
1 \to \un{F^\ur}_\fp(\hat{\CO}_\fp^\sh) \to \un{T^\ur}^0_\fp(\hat{\CO}_\fp^\sh) \to \un{(T')^\ur}^0_\fp(\hat{\CO}_\fp^\sh) \to 1.\qedhere 
\end{equation*}
\end{proof}

\begin{cor} \label{surjection of local universal cover}
The homomorphism of $\hat{\CO}_\fp$-schemes $\un{G^\sc}_\fp \to \un{G}_\fp^0$ is surjective. 
\end{cor}

\begin{proof}
Our assumption $(\text{char}(K),F)=1$ in Section \ref{section1} 
implies that the isogeny $\pi_\fp: T^\sc_\fp \to T_\fp$ is separable at any $\fp$.  
As $\un{G}_\fp^0(\hat{\CO}_\fp^\sh) = \un{T}_\fp^0(\hat{\CO}_\fp^\sh) \fX(\hat{\CO}_\fp^\sh)$, 
the surjection of groups of $\hat{\CO}_\fp^\sh$-points in Proposition \ref{tori surjection} can be extended to $\un{\pi}:\un{G^\sc}^0_\fp(\hat{\CO}_\fp^\sh) \surj \un{G}_\fp^0(\hat{\CO}_\fp^\sh)$. 
As $G_\fp^\sc$ is simply connected, $\un{G^\sc}_\fp$ has a connected special fiber (see \cite[\S~3.5.2]{Tits}).   
By \cite[Proposition~1.7.6]{BT2}, we know that the coordinate ring representing $\un{G}_\fp$ is 
$$ \hat{\CO}_\fp[\un{G}_\fp] = \left\{ f \in \hat{K}_\fp[G_\fp] : f(\un{G}_\fp(\CO^\sh_\fp)) \subset \CO^\sh_\fp \right \} \subset \hat{K}_\fp[G_\fp].$$
As $\un{\pi}(\un{G^\sc}_\fp(\hat{\CO}_\fp^\sh)) = \un{G}^0_\fp(\hat{\CO}_\fp^\sh)$,   
any function $f \in \hat{\CO}_\fp[\un{G}^0_\fp]$ satisfies
$$ f \circ \un{\pi}(\un{G^\sc}_\fp(\hat{\CO}_\fp^\sh)) \subset f(\un{G}_\fp(\hat{\CO}_\fp^\sh)) \subset \hat{\CO}_\fp^\sh, $$  
thus $f \circ \un{\pi} \in \hat{\CO}_\fp[\un{G^\sc}_\fp]$ yielding the surjection of the contravariant functor of schemes. 
\end{proof}

\begin{lem} \label{volumes of compact groups}
$\varpi_\fp(\un{G^\sc}_\fp(\hat{\CO}_\fp)) = \varpi_\fp(\un{G}_\fp^0(\hat{\CO}_\fp))$. 
\end{lem}

\begin{proof}
Consider the following exact sequences, obtained by the reduction of groups of points
\begin{align*} 
1 \to \un{G^\sc}_\fp^1(\hat{\CO}_\fp) \to \un{G^\sc}_\fp^0(\hat{\CO}_\fp) &\stackrel{\red}{\longrightarrow} \ov{G^\sc}_\fp^0(k_\fp) \to 1, \\ \nonumber 
1 \longrightarrow \un{G}_\fp^1(\hat{\CO}_\fp) \longrightarrow \un{G}_\fp^0(\hat{\CO}_\fp) &\stackrel{\red}{\longrightarrow} \ov{G}_\fp^0(k_\fp) \to 1. \nonumber 
\end{align*}
Both kernels are homeomorphic to $\fp^{\dim G_\fp}$ (cf.\ the proof of Proposition \ref{volume of identity component}) 
thus sharing both the same volume with respect to $\varpi_\fp$, namely $q^{-\dim G_\fp}$.  
Further, as the residue field $k_\fp$ is finite and the reductions $\ov{G^\sc}_\fp^0 = \ov{G^\sc}_\fp$ 
and $\ov{G}_\fp^0$ are connected and $k_\fp$-isogeneous, 
they share the same number of rational $k_\fp$-points (see \cite[\S~16.8]{Bor}). 
Now the claim follows from Proposition \ref{volume of identity component}.  
\end{proof}

\bk


\section{The Tamagawa number of semisimple groups} \label{section4}
We return to the definition of $G$ over the global field $K$ as introduced in Section~\ref{section1}. 
Let $\om$ be a differential $K$-form on $G$ of highest degree. 
It induces a Haar measure on the adelic group $G(\A)$ of $G$, 
which is unique up to a scalar multiplication. 
Let $\om_\fp$ be the multiplicative Haar measure induced locally by $\om$ at $\fp$.  
The \emph{Tamagawa measure} on $G(\A)$ is defined as 
$$ \tau = q^{-(g-1)\dim G} \prod_\fp \om_\fp $$
where $g$ is the genus of $C$. 

\begin{remark} \label{local measure up to scalar mult}
Since the multiplicative Haar measure on $G_\fp(\hat{K}_\fp)$ at any $\fp$ is unique up to a scalar multiplication, 
there exists $\lm_\fp \in \hat{K}_\fp^\times$ such that $\om_\fp = \lm_\fp\varpi_\fp$ 
(see notation in Section \ref{section3})      
and so Lemma \ref{volumes of compact groups} remains true after replacing $\varpi_\fp$ with $\om_\fp$.   
\end{remark}

Due to the product formula the measure $\tau$ does not depend on the choice of $\om$, 
i.e., for each $\lambda \in K^\times$ the volume forms $\om$ and $\lm\om$ 
yield identical Haar measures (cf.\ \cite[2.3.1]{Weil}). 
Therefore $\tau$ is well defined and the following quantity is meaningful. 
Identifying $K$ with its diagonal embedding in $\A$ and consequently $G(K)$ with its diagonal embedding in $G(\A)$, 
we consider the following arithmetic invariant of $G$:  

\begin{definition} \label{Tamagawa number semisimple groups}
The \emph{Tamagawa number} $\tau(G)$ of $G$ is the volume of $G(\A)/G(K)$ 
with respect to the Tamagawa measure $\tau$.
\end{definition} 

Recall that all discrete valuations of $K$ are non-archimedean. 
For any finite set $S$ of primes of $K$, the \emph{ring of $S$-adeles} is:
\begin{equation*}
\A_S := \left \{ (x_\fp)_{\fp \notin S} : 
x_\fp \in \hat{\CO}_\fp \ \text{for almost all} \ \fp \right \} \subset \prod\limits_{\fp \notin S} \hat{K}_\fp. 
\end{equation*}
We also define: 
\begin{equation*} 
\A(S) := \prod_{\fp \in S} \hat{K}_\fp \times \prod_{\fp \notin S} \hat{\CO}_\fp. 
\end{equation*}

For any prime $\fp$ let $\un{G}_\fp(\hat{\CO}_\fp)$ be the maximal subgroup of $G_\fp(\hat{K}_\fp)$ 
w.r.t. some special point $x$ as defined in (\ref{Local invariants}).  
With the notation as in Section~\ref{section1} we set 
\begin{equation*} 
G_S := \prod_{\fp \in S} G_\fp(\hat{K}_\fp), \ \ G(\A(S)) := G_S \times \prod_{\fp \notin S} \un{G}_\fp(\hat{\CO}_\fp). 
\end{equation*}

\begin{definition}[{\cite[p.~187]{Kneser:1966}, \cite{Platonov:1969}}] 
We say that $G$ satisfies the \emph{strong approximation} property w.r.t. a finite set of primes $S$, 
if the diagonal embedding $G(K) \hookrightarrow G(\A_S)$ is dense, 
or, equivalently, if $G_S \cdot G(K)$ is dense in $G(\A)$. 
If $|S|=1$ we call it the \emph{absolute strong approximation} property. 
\end{definition}

\begin{theorem}[{\cite[Theorem A]{Pr1}}] \label{non-compact}
Let $G$ be a simply connected $K$-group. 
If the topological group $G_S$ is non-compact w.r.t. to a finite set of primes $S$, 
then $G_S \cdot G(K)$ is dense in $G(\A)$.   
\end{theorem}

\begin{theorem} [{\cite[Thm.~3.2~3)]{Tha}, \cite[Prop.~8.8]{PR} in the number field case}] \label{class group}   
Let $G$ be a connected reductive $K$-group such that the simply connected
covering of the derived subgroup of $G$ has the strong approximation property w.r.t. a finite set of primes $S$. 
Then $G(\A(S))G(K)$ is a normal subgroup of $G(\A)$ with finite abelian quotient, 
the {\em $S$-class group} $\mathrm{Cl}_S(G) = G(\A)/G(\A(S))G(K)$ of cardinality 
$h_S(G) = |\mathrm{Cl}_S(G)|$. 
\end{theorem}

We choose an arbitrary closed point $\iy$ of $C$ to be the point at infinity, and define: 
$$ \A_\iy := \A(\{\iy\}), \ \ G(\A_\iy) := G_\iy(\hat{K}_\iy) \times \prod_{\fp \neq \iy} \un{G}_\fp(\hat{\CO}_\fp). $$   
The following facts are now deduced from the preceding Theorems in the case of $S=\{\iy\}$: 

\begin{definition} \label{class number of semisimple groups}
There exists a finite set $\{ x_1,..., x_h \} \subset G(\A)$ such that
$$ G(\A) = \bigsqcup\limits_{i=1}^h G(\A_\iy) x_i G(K).$$  
The finite number $h=h_\iy(G)$ is called the \emph{class number} of $G$ 
(see \cite[Satz~7]{Behr:1969}, \cite[Prop.~3.9]{BP}, also \cite[proof of Theorem~2.1]{Bux/Wortman:2011}). 
\end{definition}

\begin{remark} \label{Gsc strong approximation}
As our group $G$ is assumed to be locally isotropic everywhere,  
by Theorem \ref{non-compact} in the case of $S=\{\iy\}$, 
$G^\sc$ admits the absolute strong approximation property implying $h_\iy(G^\sc)=1$. 
\end{remark}

According to Theorem \ref{class group} together with Remark \ref{Gsc strong approximation},  
$G(\A_\iy)G(K)$ is a normal subgroup of $G(\A)$ 
and we may consider the natural epimorphism: 
$$ \varphi: G(\A) / G(K) \surj G(\A) / G(\A_\iy)G(K): \ \forall x \in G(\A): x G(K) \mapsto x G(\A_\iy)G(K) $$
for which 
$$ \ker(\varphi) = \{ x G(K) : x \in G(\A_\iy)G(K) \} = G(\A_\iy)G(K)/G(K) \cong G(\A_\iy) / G(\A_\iy) \cap G(K).  $$
Since all fibers of $\varphi$ are isomorphic to $\ker(\varphi)$ we get a bijection of measure spaces   
\begin{eqnarray} \label{decompositionfirst}
G(\A)/G(K) &\cong & \im(\varphi) \times \ker(\varphi) \\ \notag
           & = & \left( G(\A) / G(\A_\iy)G(K) \right )  \times \left( G(\A_\iy) / G(\A_\iy) \cap G(K) \right) \\ \notag
           &\stackrel{\ref{class group}}{\cong}& \mathrm{Cl}_\iy(G) \times \left( G(\A_\iy)/ \G \right) 
\end{eqnarray}
on which the first factor cardinality is $h_\iy(G)$ and $\G := G(\A_\iy) \cap G(K)$.  
We will next study the volume of the second factor. 

\bk


\section{On the cohomology of $\Oiy$-schemes and relative local covolumes} \label{section5}

The discrete group $\G = G(K) \cap G(\A_\iy)$ consists only of elements over the ring of $\{\iy\}$-integers of $K$, namely:
$$ \Oiy = \{ a \in K \mid v_\fp(a) \geq 0 \ \ \forall \fp \neq \iy \}. $$
So it would be natural to describe it using an $\Oiy$-scheme. Consider its following construction: 
For any $\fp$ let $\wt{G}_\fp$ be the Bruhat-Tits model of $G_\fp$ defined over $\CO_\fp$, i.e. such that: 
\begin{enumerate}
	\item $\wt{G}_\fp \otimes_{\CO_\fp} \hat{K}_\fp = G_\fp$, and: 
	\item $\wt{G}_\fp \otimes_{\CO_\fp} \hat{\CO}_\fp = \un{G}_\fp$. 
\end{enumerate}
According to Proposition D.4(a) in \cite[\S~6.2]{BLR} the patch $(G_\fp,\un{G}_\fp,\tau)$,  
where $\tau$ is the canonical isomorphism $G_\fp \otimes_{K_\fp} \hat{K}_\fp \cong \un{G}_\fp \otimes_{\CO_\fp} \hat{K}_\fp$, 
corresponds uniquely to the $\CO_\fp$-module $\wt{G}_\fp$, in the sense that it covers it with a canonical descent datum. 
Now since $C$ is one dimensional, for any two distinct primes $\fp_1$ and $\fp_2$, 
the product $\CO_{\fp_1} \otimes \CO_{\fp_2}$ is isomorphic to $K$. 
Thus we may Zariski-glue all geometric fibers $\{\Sp \CO_\fp : \fp \neq \iy \}$ along the generic point $\Sp K$, resulting in $\Sp \Oiy$. 
Then the aforementioned patches cover (with descent datum) a unique group scheme $\CG$ over $\Sp \Oiy$. 
Moreover, for any $\fp$, the localization $(\Oiy)_\fp$ is a base change of $\CO_\fp$. 
Thus the bijection $\Sp (\Oiy)_\fp \to \Sp \CO_\fp$ is faithfully flat (see \cite[Thm.~3.16]{Liu}). 
Hence $\CG$ extended to $\Sp \hat{\CO}_\fp$ is smooth by construction so that $\CG$ is smooth at $\fp$ by faithfully flat descent, see \cite[17.7.3]{EGAIV}. 
Its generic fiber is $G$ and it satisfies:
\begin{equation*} 
\CG(\Oiy) = G(K) \cap \prod_{\fp \neq \iy} \un{G}_\fp(\hat{\CO}_\fp) = G(K) \cap G(\A_\iy). 
\end{equation*} 

We denote by $\CG^0$ the subscheme of $\CG$ whose geometric fibers are $\un{G}^0_\fp$. 
The same construction for $G^\sc$ is denoted by $\CG^\sc$. 
The surjectivity at the geometric fibers $\un{G^\sc}_\fp \to \un{G}_\fp^0$ 
(see Lemma~\ref{surjection of local universal cover})  
leads to the surjection $\pi_{\Oiy}:\CG^\sc \to \CG^0$ over $\Sp \CO_{\{\iy\}}$ as \'etale sheaves. 
As we assumed $\text{char}(K)$ to be prime to $|F|$, 
$\F:= \ker[\CG^\sc \to \CG^0]$ is smooth as well 
and we have an exact sequence of $\Oiy$-models: 
\begin{equation} \label{exact sequence of semisimple groups models}
1 \to \F \to \CG^\sc \to \CG^0 \to 1. 
\end{equation}   

Let $\T$ be the subscheme of $\CG$ whose generic fiber is $T$, 
let $\T^0 := \T \cap \CG^0$ and let $\T^\sc$ be its preimage under $\pi_{\Oiy}$ in $\CG^\sc$. 
Being central (as all its geometric fibers), $\F$ is equal to the kernel of the corresponding $\Oiy$-tori-models, 
fitting into the following exact sequence of $\Oiy$-schemes 
\begin{equation} \label{exact sequence of tori models}
1 \to \F \to \T^\sc \to \T^0 \to 1. 
\end{equation}

\begin{lem}\label{H1(Gsc) trivial2} 
$H^1_\et(\Oiy,\CG^\sc)=1$. 
\end{lem}

\begin{proof} 
According to Nisnevitch (\cite[3.6.2]{Nis}) we have the following exact sequence 
\begin{equation*} 
1 \to \mathrm{Cl}_\iy(G^\sc) \to H^1_\et(\Oiy,\CG^\sc) \to H^1(K,G^\sc(K^s)) 
\end{equation*}
on which in our case $\mathrm{Cl}_\iy(G^\sc)$ is trivial (see Remark \ref{Gsc strong approximation}) 
and the latter group in the sequence is trivial as well due to Harder's result (see \cite[Satz~A]{Harder:1975}).  
The claim follows.  
\end{proof}

\begin{lem} \label{Fiscoh2} 
Let $\pi_{\Oiy} : \CG^\sc(\Oiy) \to \CG^0(\Oiy)$. Then: 
$$ 
j_\iy(G):=\fc{|\cok(\pi_{\Oiy})|}{|\ker(\pi_{\Oiy})|} = \fc{|H^1_\et(\Oiy,\F)|}{|\F(\Oiy)|}. 
$$
If in particular $G$ is quasi-split and $C$ is of genus $g=0$, 
then $j_\iy(G)=1$ which means that the discrete groups $\G^\sc$ and $\G^0$ are bijective.  
\end{lem}

\begin{proof}
Since $\F$ is smooth as well as its geometric fibers we have: $H^1_\et(\Oiy,\F) = H^1_\text{fppf}(\Oiy,\F)$.   
Due to Lemma \ref{H1(Gsc) trivial2}, 
flat cohomology applied on sequence (\ref{exact sequence of semisimple groups models}) 
gives rise to the following sequence of groups of $\Oiy$-points:
\begin{equation} \label{semisimple OK-points sequence}
1 \to \F(\Oiy) \to \CG^\sc(\Oiy) \stackrel{\pi_{\Oiy}}{\to} \CG^0(\Oiy) \to H^1_\et(\Oiy,\F) \to 1. 
\end{equation} 
This gives us the first assertion. \\
If $G$ is quasi-split, then $T^\sc$ is quasi-trivial, i.e. isomorphic to a
finite product of Weil's tori $R_{K_1/K}(\BG_m) \times \cdots \times R_{K_n/K}(\BG_m)$ 
where the $K_i$'s are finite separable extensions of $K$ (see Remark \ref{quasi-trivial Cartan}). 
Consequently, its $\Oiy$-model $\T^\sc$ is isomorphic to $R_{\CO_{\{\iy\},1}} \times \cdots \times R_{\CO_{\{\iy\},n}}$ 
where $\CO_{\{\iy\},i}$ stands for the integral closure of $\Oiy$ in $K_i$ for each $i$. 
By Shapiro's formula for the flat topology, we have: 
$$ 
H^1_{\text{fppf}}(\Oiy,\T^\sc) \cong \bigoplus_{i=1}^n H^1_{\text{fppf}}(\Oiy,R_{\CO_{\{\iy\},i}/\Oiy}(\BG_m)) 
                               \cong \bigoplus_{i=1}^n H^1_{\text{fppf}}(\CO_{\{\iy\},i},\BG_m)
                               =     \bigoplus \Pic(C^\af).            
$$
If in addition the curve $C$ is of genus $0$ we have $\Pic(C^\af)=0$, 
and so flat cohomology applied on sequence (\ref{exact sequence of tori models}) 
gives rise to the following sequence of multiplicative groups of $\Oiy$-points 
\begin{equation} \label{tori OK-points sequence}
1 \to \F(\Oiy) \to \T^\sc(\Oiy) \stackrel{\pi}{\to} \T^0(\Oiy) \to H^1_\et(\Oiy,\F) \to 1. 
\end{equation}
Recall that $\Oiy = K \cap \bigcap_{\fp \neq \iy} \CO_\fp$, 
i.e., $\Oiy$ consists of exactly those elements of $K$ that do not have poles at any place $\fp \neq \iy$. 
If $x \in \CO_{\{\iy\}}$ has a proper pole at $\iy$, then it has a proper zero at some place $\fp \neq \iy$. 
Hence its inverse $x^{-1} \in K$ has a proper pole at that place and, thus, $x^{-1} \in K \backslash \CO_{\{\iy\}}$. 
We conclude that the only invertible elements of $\Oiy$ are the constants.
In other words, since the curve $C$ is projective, its regular functions are exactly the constants. 
This means that $\T^\sc(\Oiy) = \T^\sc(\BF_q)$ and $\T(\Oiy) = \T(\BF_q)$ are finite groups. 

As the reduction of all geometric fibers of $\T^\sc$ and $\T^0$ are smooth and connected, 
the specializations 
$\ov{\T^\sc} = \T^\sc \otimes_{\Sp \Oiy} \Sp \BF_q$ and $\ov{\T}^0 = \T^0 \otimes_{\Sp \Oiy} \Sp \BF_q$ 
are connected $\BF_q$-schemes,
where $\Sp \BF_q \to \Sp \Oiy$ is the closed immersion of the special point. 
Thus the exact  sequence (\ref{tori OK-points sequence}) can be rewritten as: 
\begin{equation} \label{tori Fq-points sequence}
1 \to \F(\Oiy) \to \ov{\T^\sc}(\BF_q) \stackrel{\pi}{\to} \ov{\T}^0(\BF_q) \to H^1_\et(\Oiy,\F) \to 1. 
\end{equation}  
The surjectivity of $\T^\sc \to \T^0$ implies the one of $\ov{\T^\sc} \to \ov{\T}^0$.  
These schemes are isogenous, connected and defined over $\BF_q$, so they 
share the same number of $\BF_q$-points.   
Then the exactness of (\ref{tori Fq-points sequence}) implies that $|\F(\Oiy)| = |H^1_\et(\Oiy,\F)|$.  
Returning back to the exact sequence (\ref{semisimple OK-points sequence}) we get the claim.      
\end{proof}


The group $G(\A)$ admits a natural action on the product 
$\B = \prod_\fp \B_\fp$ of the Bruhat--Tits buildings, 
and its subgroup $G(\mathbb{A}_\iy)$ fixes the fundamental special vertex 
of each building $\B_\fp$ with $\fp \neq \iy$. 
Identifying $\B_\iy$ with its product with these fundamental special vertices 
therefore yields an action of $G(\A_\iy)$ on $\B_\iy$. 
Let: 
$$ G^0(\A_\iy) = G_\iy(\hat{K}_\iy) \times \prod_{\fp \neq \iy} \un{G}_\fp^0(\hat{\CO}_\fp), \ \G^0 = G^0(\A_\iy) \cap G(K) \subset \G. $$
Notice that as $G^\sc$ is simply connected $\G^\sc := G^\sc(\A_\iy) \cap G(K) = (\G^\sc)^0$. \\ 
Consider the following compact subgroups:
\begin{align*} 
U^\sc = \prod_\fp \un{G}^\sc_\fp(\hat{\CO}_\fp) \subset G^\sc(\A_\iy), \ 
    U = \prod_\fp \un{G}^0_\fp(\hat{\CO}_\fp) \subset G(\A_\iy). 
\end{align*}
Let $Y^\sc$ and $Y$ be the sets of representatives respectively for the double cosets sets:  
\begin{align} \label{Ysc and Y}      
\G^\sc                    \backslash G^\sc(\A_\iy)      / U^\sc \cong 
 (\G^\sc \cap G^\sc(\hat{K}_\iy)) \backslash G^\sc_\iy(\hat{K}_\iy) / \un{G}^\sc_\iy(\hat{\CO}_\iy), \\ \nonumber
\G^0                      \backslash G(\A_\iy)          / U \cong 
   (\G^0 \cap G(\hat{K}_\iy))     \backslash G_\iy(\hat{K}_\iy)     / \un{G}^0_\iy(\hat{\CO}_\iy). 
\end{align}
For any $y \in Y^\sc, yU^\sc y^{-1}$ is compact and $\G^\sc$ is discrete thus their intersection is finite. 
More precisely, by the isomorphism above any such $y$ may represents a non-trivial double coset only at its $\iy$-component, 
whence $yU^\sc y^{-1} \subset G^\sc(\A_\iy)$ and therefore: 
\begin{align*}
yU^\sc y^{-1} \cap \G^\sc = yU^\sc y^{-1} \cap (G^\sc(K) \cap G^\sc(\A_\iy)) 
                                 = yU^\sc y^{-1} \cap G^\sc(K). 
\end{align*}
But conjugation by $y$ on the $\iy$-component of $U$ is a shift to the stabilizer of $yx$ in $G_\iy(\hat{K}_\iy)$: 
$$ yU^\sc y^{-1} = \un{G}^\sc_{\iy,yx}(\hat{\CO}_\iy) \times \prod_{\fp \neq \iy} \un{G}^\sc_\fp(\hat{\CO}_\fp). $$  
Thus $yU^\sc y^{-1} \cap G^\sc(K)$ admits an underlying group scheme $\wt{G}^\sc$ having only global sections on $K$, 
i.e. defined over $\Sp \BF_q$ (recall that $C$ is projective). 
We denote by $\wt{G}_{\pi(y)}$ the resulting $\BF_q$-group for the same construction for $G$ with $\pi(y)$ 
(here again: 
$\pi(y) U \pi(y)^{-1} =  \un{G}^0_{\iy,\pi(y)x}(\hat{\CO}_\iy) \times \prod_{\fp \neq \iy} \un{G}^0_\fp(\hat{\CO}_\fp)$.) 
The surjectivity of  
$\un{G}^\sc_{\iy,yx} \surj \un{G}^0_{\iy,\pi(y)x}, \un{G}^\sc_\fp \surj \un{G}^0_\fp \ \forall \fp \neq \iy$ 
and $G^\sc \surj G$ 
implies the one of $\wt{G}_y^\sc \surj \wt{G}_{\pi(y)}$ having a finite kernel as well. 
So the groups $\wt{G}_y^\sc$ and $\wt{G}_{\pi(y)}$,  
being isogeneous, connected and of finite dimension, defined over the finite field $\BF_q$, 
share the same finite number of $\BF_q$-points, 
i.e.: 
\begin{equation} \label{equality of Stab intersections} 
\forall y \in Y^\sc: |yU^\sc y^{-1} \cap \G^\sc| = |\pi(y) U \pi(y)^{-1} \cap \G^0|. 
\end{equation}
As $G(\A_\iy)$ is unimodular by \cite[Corollary~I.2.3.3]{Margulis:1991} 
we get to Serre's formula (\cite[p.~84]{Ser1},\cite[Corollary~1.6]{Bass/Lubotzky:2001}) : 
\begin{eqnarray} \label{Tamagawa number semisimple formula}
\tau(G) & \stackrel{(\ref{decompositionfirst})}{=} & h_\iy(G) \cdot \tau(G(K) \backslash G(\A_\iy)) 
          = h_\iy(G) \cdot \sum\limits_{y \in Y} \tau(yU) \\ \notag
        & = & h_\iy(G) \cdot \sum\limits_{y \in Y} \fc{\tau(U)}{|yUy^{-1} \cap \G^0|} \\ \notag
        & = & q^{-(g-1)\dim(G)} \cdot h_\iy(G) \cdot \prod_\fp \om_\fp(\un{G}^0_\fp(\hat{\CO}_\fp)) 
                                               \cdot \sum_{y \in Y} \fc{1}{|y U y^{-1} \cap \G^0|} \\ \notag
        & = & q^{-(g-1)\dim(G)} \cdot h_\iy(G) \cdot i_\iy(G) \cdot \prod_\fp \om_\fp(\un{G}^0_\fp(\hat{\CO}_\fp))    
\end{eqnarray} 
where $i_\iy(G)=\sum_{y \in Y} \fc{1}{|y U y^{-1} \cap \G^0|}$.

\begin{lem} \label{covolume equality}
With the previously introduced notations one has
\begin{equation*}
\fc{i_\iy(G)}{i_\iy(G^\sc)} = \fc{t_\iy(G)}{j_\iy(G)}.  
\end{equation*}
\end{lem}

\begin{proof}
Let us regard the double cosets groups in formulas \ref{Ysc and Y} represented by $Y^\sc$ and $Y$ respectively.  
The representatives of $G^\sc(\A_\iy) / U^\sc \cong G^\sc_\iy(\hat{K}_\iy) / \un{G}^\sc_\iy(\hat{\CO}_\iy)$ 
and of $G(\A_\iy) / U \cong G_\iy(\hat{K}_\iy) / \un{G}^0_\iy(\hat{\CO}_\iy)$ 
correspond to vertices in the orbits of $x$ in $\B_\iy$ obtained by the actions of 
$G^\sc_\iy(\hat{K}_\iy)$ and $G(\hat{K}_\iy)$ respectively (the Iwahori subgroup is the kernel of this action in each case). 
Since these actions are transitive on the alcoves in $\B_\iy$, 
it is sufficient to compare the orbits inside each alcove.  
Recall that $t_\iy(G)$ is the number of (special) points in the orbit of $x$ in one alcove. 
For each $y \in Y^\sc$, let $x_1,...,x_{t_\iy(G)}$ be the types representatives 
in the $\un{G}_\iy(\hat{K}_\iy)$-orbit of $x_1 = \pi(y)$. 
Since the $G_\iy^\sc(\hat{K}_\fp)$-action on $B_\iy$ is type preserving,  
this correspondence of types in the $x$-orbit is one to one. 

To accomplish the comparison between $Y^\sc$ and $Y$, 
the above right quotients, taken modulo the discrete subgroups $\G^\sc$ and $\G^0$ (from the left) respectively, 
correspond to vertices in some fundamental domains of the aforementioned orbits of $x$. 
In Lemma \ref{Fiscoh2} we compared between these subgroups and got that $\G^0$ is bijective to $j_\iy(G)$ times $\G^\sc$.  
Moreover, along any orbit the Bruhat-Tits schemes are isomorphic (see \cite[\S~2.5.,p.~47]{Tits}) 
and have isomorphic reductions. 
Thus: $\forall i: \ |x_i U x_i^{-1} \cap \G^0| = |\pi(y) U \pi(y)^{-1} \cap \G^0|$. 
We get: 
\begin{align*}
i_\iy(G) & = \sum_{y \in Y} \fc{1}{|x_i U x_i^{-1} \cap \G^0|}  \notag \\
         & = \sum_{y \in Y} \fc{1}{|\pi(y) U \pi(y)^{-1} \cap \G^0|}  \notag \\
         & \stackrel{(\ref{equality of Stab intersections})}{=}  
\fc{t_\iy(G)}{j_\iy(G)} \cdot \sum_{y \in Y^\sc} \fc{1}{|y U^\sc y^{-1} \cap \G^\sc|} \\ \notag
         & = \fc{t_\iy(G)}{j_\iy(G)} \cdot i_\iy(G^\sc). \qedhere 
\end{align*}
\end{proof}

Recall that
\begin{eqnarray*}
\tau(G) 
  & \stackrel{(\ref{Tamagawa number semisimple formula})}{=} 
  & q^{-(g-1)\dim G} \cdot h_\iy(G) \cdot i_\iy(G) \cdot \prod_\fp \om_\fp(\un{G}^0_\fp(\hat{\CO}_\fp)). 
\end{eqnarray*}
Clearly the invariant $q^{-(g-1)\dim G}$ is the same for both $G$ and $G^\sc$, 
as well as the volume of the compact subgroups (see Lemma \ref{volumes of compact groups} 
and Remark \ref{local measure up to scalar mult}). We conclude that: 
\begin{align*}
\fc{\tau(G)}{\tau(G^\sc)} = \fc{h_\iy(G)}{h_\iy(G^\sc)} \cdot \fc{i_\iy(G)}{i_\iy(G^\sc)}. 
\end{align*}
Now assuming the validity of the Weil Conjecture: $\tau(G^\sc)=1$ and due to the strong approximation related to $G^\sc$  
for which $h_\iy(G^\sc)=1$ (see Remark \ref{Gsc strong approximation}), plus Lemma \ref{covolume equality} we finally deduce that: 

\begin{cor} \label{ht}
$$ \tau(G) = h_\iy(G) \cdot \fc{t_\iy(G)}{j_\iy(G)}. $$ 
If $G$ is quasi-split and $C$ is of genus $0$, according to Lemma \ref{Fiscoh2} this formula simplifies to:
$$ \tau(G) = h_\iy(G) \cdot t_\iy(G). $$ 
\end{cor}

\bk

%


\section{Number of types in the orbit of a special point} \label{section6}

We retain the notation and terminology introduced in the preceding sections. 

\begin{lem} \label{H1=1}
For any prime $\fp$ one has
$H^1(\la \s_\fp \ra,\pi(G_\fp^\sc(\hat{K}_\fp^\ur)))=1$. 
\end{lem}

\begin{proof}
At any prime $\fp$ we may consider the following exact sequence of $\hat{K}_\fp$-groups:
$$ 1 \to F_\fp \to G_\fp^\sc \to \pi(G^\sc) \to 1. $$
Due to Harder \cite[Satz~A]{Harder:1975} we know that $H^1(\la \s_\fp \ra,G_\fp^\sc(\hat{K}_\fp^\ur))=1$, 
hence $\la \s_\fp \ra$-cohomology gives rise to the exact sequence:
$$ 1 \to H^1(\la \s_\fp \ra,\pi(G_\fp^\sc(\hat{K}_\fp^\ur)) \to H^2(\la \s_\fp \ra,F_\fp(\hat{K}_\fp^\ur)) $$
on which the right term is trivial as $F_\fp(\hat{K}_\fp^\ur)$ is finite. 
This gives the required result.   
\end{proof}

\begin{lem} \label{what is k} 
The number $t_\iy(G)$ of (special) types in the $\un{G}_\iy(\hat{K}_\iy)$-orbit of the
fundamental special vertex $x$ in $\B_\iy$ is given by 
$$ t_\iy(G) = |H^1(I_\iy,F_\iy(\hat{K}_\iy^s))^{\s_\iy}| = |\wh{F_\iy}^{\fg_\iy}|. $$ 
\end{lem}

\begin{proof}
Galois $I_\iy$ and $\fg_\iy$-cohomology yield the exact diagram
$$\xymatrix{
1 \ar[r] & F_\iy(\hat{K}_\iy^\ur) \ar[r] & G^\sc_\iy(\hat{K}_\iy^\ur) \ar[r]^\pi  & G_\iy(\hat{K}_\iy^\ur) \ar[r]    & H^1(I_\iy,F_\iy(\hat{K}_\iy^s)) \ar[r] & 1 \\
1 \ar[r] & F_\iy(\hat{K}_\iy)     \ar[r] & G^\sc_\iy({\hat{K}_\iy})   \ar[r]^\pi \ar@{^{(}->}[u] & G_\iy(\hat{K}_\iy) \ar[r] \ar@{^{(}->}[u] & H^1(\fg_\iy,F_\iy(\hat{K}_\iy^s)) \ar[r] \ar[u] & 1
}$$
The group $\pi_\iy(G^\sc_\iy(\hat{K}_\iy^\ur)) \cap G_\iy({\hat{K}_\iy})$ is the largest type-preserving subgroup of $G_\iy(\hat{K}_\iy)$. 
By the classification of affine Dynkin diagrams
an automorphism of $\B_\iy$ preserves the types of special vertices in $\B_\iy$ 
if and only if it preserves types of arbitrary vertices. 
Therefore the cosets of $\pi_\iy(G^\sc_\iy(\hat{K}_\iy^\ur)) \cap G_\iy(\hat{K}_\iy)$ in $G_\iy(\hat{K}_\iy)$ are in $1$-to-$1$ 
correspondence with the types of special vertices in the $\un{G}_\iy(\hat{K}_\iy)$-orbit. 
We conclude that 
$$ t_\iy(G) = \left|G_\iy(\hat{K}_\iy)/\left(\pi(G^\sc_\iy(\hat{K}_\iy^\ur)) \cap G_\iy(\hat{K}_\iy)\right)\right|. $$
The exact sequence
$$ 1 \to F_\iy(\hat{K}_\iy^\ur) \to G_\iy^\sc(\hat{K}_\iy^\ur) \stackrel{\pi}{\rightarrow} G_\iy(\hat{K}_\iy^\ur) \to H^1(I_\iy,F_\iy(\hat{K}_\iy^s)) \to 1 $$
can be shortened to
$$ 1 \to \pi(G_\iy^\sc(\hat{K}_\iy^\ur)) \to G_\iy(\hat{K}_\iy^\ur) \to H^1(I_\iy,F_\iy(\hat{K}_\iy^s)) \to 1. $$
Applying $\la \s_\iy \ra$-cohomology on this exact sequence gives the exact sequence 
$$
1 \to \pi(G_\iy^\sc(\hat{K}_\iy^\ur)) \cap G_\iy(\hat{K}_\iy) \to G_\iy(\hat{K}_\iy) \to H^1(I_\iy,F_\iy(\hat{K}_\iy^s))^{\s_\iy} 
                                                                   \to H^1(\la \s_\iy \ra,\pi(G_\iy^\sc(\hat{K}_\iy^\ur))) 
$$
on which the right-hand group is trivial by Lemma \ref{H1=1}. Hence
$t_\iy(G) = |H^1(I_\iy,F_\iy(\hat{K}_\iy^s))^{\s_\iy}|$. 

More explicitly, the Kottwitz epimorphism together with Galois descent, yields an epimorphism $T_\iy(\hat{K}_\iy) \to X_*(T_\iy)_{I_\iy}^{\s_\iy}$ 
whose kernel is the Iwahori subgroup $\un{T}^0_\iy(\hat{\CO}_\iy)$ (see \cite[Corollary~3.2]{Bit}). 
We get the following exact diagram
$$\xymatrix{
1 \ar[r] & \un{F}_\iy(\hat{\CO}_\iy) \ar[r] \ar@{=}[d] &  \un{T^\sc}_\iy(\hat{\CO}_\iy)     \ar[r]^{\un{\pi}_\iy} \ar[d] & \un{T}^0_\iy(\hat{\CO}_\iy) \ar[r] \ar[d] & H^1(\la \s_\iy \ra,\un{F}_\iy(\hat{\CO}_\iy^\sh))  \ar[d] \ar[r] & 0 \\ 
1 \ar[r] & F_\iy(\hat{K}_\iy)   \ar[r] \ar[d] &  T_\iy^\sc(\hat{K}_\iy) \ar[r]^{\pi_\iy}   \ar[d] & T_\iy(\hat{K}_\iy) \ar[r] \ar[d] & H^1(\fg_\iy,F_\iy(\hat{K}_\iy^s)) \ar[r] \ar[d] & 0 \\
         &    0           \ar[r] & X_*(T_\iy^\sc)_{I_\iy}^{\s_\iy} \ar[r]^{\pi^\vee_{I_\iy}} & X_*(T_\iy)_{I_\iy}^{\s_\iy} \ar[r] & H^1(I_\iy,F_\iy(\hat{K}_\iy^s))^{\s_\iy}  \ar[r] & 0
}$$
on which the lower row can be also obtained by the following steps: 
applying the contravariant left-exact functor $\Hom(-,\Z)$ 
on the exact sequence of character $\fg_\iy$-modules
$$ 0 \to X^*(T_\iy) \to X^*(T_\iy^\sc) \to \wh{F_\iy} \to 0, $$
on which $\wh{F_\iy} = \Hom(F_\iy \otimes \hat{K}_\iy^s,\BG_{m,\hat{K}_\iy^s})$, gives the exact sequence
\begin{equation} \label{cocharacters sequence}
0 \to 0 = \Hom(\wh{F_\iy},\Z) \to  X_*(T_\iy^\sc) \stackrel{\pi^\vee}{\longrightarrow} X_*(T_\iy) \to \Ext^1(\wh{F_\iy},\Z) \to \Ext^1(X^*(T_\iy^\sc),\Z) = 0. 
\end{equation}
Applying the functor $\Hom(\wh{F_\iy},-)$ on the resolution
$$ 0 \to \Z \to \BQ \to \BQ/\Z \to 0 $$
gives rise to a long exact sequence on which as $\wh{F_\iy}$ is finite, $\Hom(\wh{F_\iy},\BQ)=0$  
and $\Ext^1(\wh{F_\iy},\BQ)=0$, showing the existence of an isomorphism
$$ \Ext^1(\wh{F_\iy},\Z) \cong \Hom(\wh{F_\iy},\BQ/\Z) = \wh{F_\iy}^* $$
where $\wh{F_\iy}^*$ is the Pontryagin dual of $\wh{F_\iy}$, 
i.e., the group of finite order characters of $\wh{F_\iy}$, see also \cite[p.~23]{Mil2}. 
Being finite, these duals are isomorphic. 
So sequence (\ref{cocharacters sequence}) can be rewritten as
\begin{equation} \label{cocharacters sequence2}
0 \to  X_*(T_\iy^\sc) \stackrel{\pi^\vee}{\longrightarrow} X_*(T_\iy) \to \wh{F_\iy}^* \to 0. 
\end{equation}
The $I_\iy$-coinvariants functor is in general only right exact, 
but here as $\un{T^\sc}_\iy$ is connected, $X_*(T_\iy^\sc)_{I_\iy}$ is free 
(see \cite[Formula~(3.1)]{Bit}) and embedded into  $X_*(T_\iy)_{I_\iy}$. 
Thus applying this functor on
$$ 0 \to  X_*(T_\iy^\sc) \stackrel{\pi^\vee}{\longrightarrow} X_*(T_\iy) \to \wh{F_\iy}^* \cong \wh{F_\iy} \to 0 $$ 
also leaves the left hand side exact
$$ 0 \to  X_*(T^\sc_\iy)_{I_\iy} \stackrel{\pi^\vee}{\longrightarrow} X_*(T_\iy)_{I_\iy} \to \wh{F_\iy}^*_{I_\iy} \to 0. $$ 
Now applying the Galois $\la \s_\iy \ra$-cohomology gives the exact lower row on the above diagram
\begin{equation} \label{t_iyG} 
0 \to  X_*(T^\sc_\iy)_{I_\iy}^{\s_\iy} \stackrel{\pi_{I_\iy}^\vee}{\longrightarrow} X_*(T_\iy)_{I_\iy}^{\s_\iy} 
  \to (\wh{F_\iy}^*_{I_\iy})^{\s_\iy} \to H^1(\la \s_\iy \ra,X_*(T^\sc_\iy)) = 0.  
\end{equation}
Returning to the diagram, as $\wh{F_\iy}^*$ being finite is isomorphic as a $\fg_\iy$-module to $\wh{F_\iy}$, we finally get
\begin{equation*}
t_\iy(G) = |H^1(I_\iy,F_\iy(\hat{K}_\iy^s))^{\s_\iy}| = |\cok(\pi^\vee_{I_\iy})| = |\wh{F_\iy}^{\fg_\iy}|. \qedhere
\end{equation*} 
\end{proof}

\begin{remark} 

\begin{enumerate}
\item Sequence (\ref{t_iyG}) illustrates the fact that the number $t_\iy(G)$ of types in the orbit of $x$ 
depends only on the embedding of $X_*(T^\sc_\iy)$ in $X_*(T_\iy)$. 
\item By the geometric version of \v{C}ebotarev's density theorem (see in \cite{Jar}), 
one may choose the point $\iy$ such that $G_\iy$ is split. 
In this case $t_\iy(G) = |F_\iy|$.  
\end{enumerate}
\end{remark}

Now Corollary \ref{ht} together with Lemma~\ref{what is k} show the Main Theorem.

\begin{maintheorem*} 
Assuming the Weil conjecture validity one has: 
$$ \tau(G) =  h_\iy(G) \cdot \fc{t_\iy(G)}{j_\iy(G)}. $$  
The number $t_\iy(G)$ satisfies 
$$t_\iy(G) = |H^1(I_\iy,F_\iy(\hat{K}_\iy^s))^{\s_\iy}| = |\wh{F_\iy}^{\fg_\iy}|$$
and: 
$$ j_\iy(G) = \fc{|H^1_\et(\Oiy,\F)|}{|\F(\Oiy)|}. $$
If in particular $G$ is quasi-split and the genus $g$ of the curve $C$ is $0$ then $j_\iy(G)=1$ and so
$$ \tau(G) =  h_\iy(G) \cdot t_\iy(G) = h_\iy(G) \cdot |\wh{F_\iy}^{\fg_\iy}|. $$  
\end{maintheorem*}

\bk


\section{Application and examples} \label{section7}
In this section we describe an application of our Main Theorem 
in case $G$ is \textbf{quasi-split} and $g=0$.  
We combine our result with \cite[Formula (3.9.1')]{Ono1} and
the techniques from \cite[\S~8.2]{PR} in order to relate the cokernels of
Bourqui's degree maps $\mathrm{deg}_{T^\sc}$ and $\mathrm{deg}_T$ from \cite[Section~2.2]{Bou}, 
where $T^\sc$ and $T$ denote suitable Cartan subgroups of $G^\sc$ and $G$ respectively;
cf.\ Proposition~\ref{D=1} below. These concrete computations will allow us to also provide a wealth of examples for which we compute the relative Tamagawa numbers.
Ono's formula was originally designed for groups over number fields 
and was generalized to the function field case in \cite[Theorem~6.1]{Behrend/Dhillon:2009}. 
We will use freely the notation concerning algebraic tori introduced in \cite{Ono1}.  In this section we will usually assume that Weil's conjecture $\tau(G^\sc) = 1$ holds.   

\begin{remark} \label{index of components}
According to the Bruhat--Tits construction $\un{G}_\fp(\hat{\CO}_\fp) = \un{T}_\fp(\hat{\CO}_\fp) \fX(\hat{\CO}_\fp)$.  
As $G_\fp$ is quasi-split one has (see \cite[Corollary~4.6.7]{BT2}) 
$\un{G}_\fp^0(\hat{\CO}_\fp) = \un{T}_\fp^0(\hat{\CO}_\fp) \fX(\hat{\CO}_\fp)$  
and so
\begin{equation*} 
[\un{G}_\fp(\hat{\CO}_\fp):\un{G}_\fp^0(\hat{\CO}_\fp)] = [\un{T}_\fp(\hat{\CO}_\fp):\un{T}_\fp^0(\hat{\CO}_\fp)].
\end{equation*} 
\end{remark}

\begin{definition}
The finite group $W(T) = T(K) \cap T^c(\A) = \T(\BF_q)$ 
is the \emph{group of units} of $T$ and its cardinality is denoted by $w(T)$. 
\end{definition}

\begin{lem} \label{wratio}
$$ 
\fc{w(T)}{w(T^\sc)} = \fc{|\T(\BF_q)|}{|\T^\sc(\BF_q)|} = \fc{|\T(\BF_q)|}{|\T^0(\BF_q)|} = [\T(\BF_q):\T^0(\BF_q)] . $$
\end{lem}

\begin{proof}
Under the assumptions of $G$ being quasi-split and $g=0$ the finite groups 
$\T^\sc(\BF_q)$ and $\T^0(\BF_q)$ are of the same cardinality (see in the proof of Lemma \ref{Fiscoh2}).  
The assertion follows. 
\end{proof}

For an algebraic $K$-torus $T$ we set the following subgroup of the adelic group $T(\A)$
\begin{equation} \label{T1}
T^1(\A) := \{x \in T(\A) : ||\chi(x)||=1 \ \forall \chi \in X^*(T)_K \}. 
\end{equation}
Let $\fg = \mathrm{Gal}(K^s/K)$. 
Following J. Oesterl\'e in \cite[I.5.5]{Oes}, D. Bourqui defines in \cite[\S 2.2.1]{Bou} the morphism
$$ \deg_T : T(\A) \to \Hom(X^*(T)^\fg,\Z) $$
with $\ker(\deg_T) = T^1(\A)$ and a finite cokernel (see \cite[Proposition 2.21]{Bou}). 
The maximal compact subgroup of $T(\A)$ is denoted by
$$ T^c(\A) := \prod_\fp \un{T}_\fp(\hat{\CO}_\fp). $$ 

\begin{definition} 
The \emph{class number} of $T$ is $ h(T) := [T^1(\A):T^c(\A)T(K)]$. 
\end{definition}

By \cite[Formula~(3.9.1')]{Ono1} for a $K$-isogeny $\pi : T \to T'$ of tori $T$, $T'$ defined over $K$ one has
\begin{equation} \label{Tamagawa isogeny}
\tau(\pi) := \fc{\tau(T')}{\tau(T)} =  \fc{w(T)}{w(T')} \fc{h(T')}{h(T)} \prod_\fp \fc{L_\fp(1,\chi_{T'_\fp}) \cdot \om_\fp(\un{T'}_\fp(\hat{\CO}_\fp))}{L_\fp(1,\chi_{T_\fp}) \cdot \om_\fp(\un{T}_\fp(\hat{\CO}_\fp))}.
\end{equation}

We shall need the following
\begin{lem} \label{H1 connected trivial1} 
Let $\un{H}_\fp$ be an affine, smooth and connected group scheme defined over $\CO_\fp$. 
Then $H^1(\la \s_\fp \ra,\un{H}_\fp(\CO_\fp^\sh))=1$. 
\end{lem}

\begin{proof} 
As $\CO_\fp$ is Henselian, we have 
$H^1(\la \s_\fp \ra,\un{H}_\fp(\CO_\fp^\sh)) \cong H^1(\la \s_\fp \ra,\ov{H}_\fp(k_\fp^s))$ 
(see Remark 3.11(a) in \cite[Chapter~III,~\S3]{Mil1}). 
The group on the right hand side is trivial by Lang's Theorem 
(see \cite{Lan} and \cite[Chapter~VI,~Proposition~5]{Ser2}).  
\end{proof}

\begin{remark} \label{H1 trivial over Kp}
As $G^\sc_\fp$ is quasi-split and simply connected, its Cartan subgroup $T^\sc_\fp$ is a quasi-trivial torus 
(i.e. a product of Weil tori). 
Thus not only $H^1(\fg_\fp,G^\sc_\fp(K_\fp^s))=1$ (which is due to Harder as aforementioned), 
but also $H^1(\fg_\fp,T^\sc_\fp(K_\fp^s))=1$ as well as $H^1(\fg,G^\sc(K^s))=1$ and $H^1(\fg,T^\sc(K^s))=1$.  
\end{remark}

\begin{lem} \label{class number ratio}
The map $\pi^\vee_K: \Hom(X^*(T^\sc)^\fg,\Z) \to \Hom(X^*(T)^\fg,\Z)$ is injective. 
One has
\begin{equation*}
   h_\iy(G) \cdot \fc{h(T^\sc)}{h(T)} = \fc{|\cok(\pi^\vee_K)|}{t_\iy(G) \cdot |\D|} \cdot 
   \fc{\prod_\fp [\un{T}_\fp(\hat{\CO}_\fp):\un{T}_\fp^0(\hat{\CO}_\fp)]}{[\T(\BF_q):\T^0(\BF_q)]}. 
\end{equation*}
\end{lem}

\begin{proof}
Since $G$ is of non-compact type, the exact sequence of $K$-groups
\begin{equation*} 
1 \to F \to G^\sc \stackrel{\pi}{\rightarrow} G \to 1
\end{equation*}
induces the exactness over the adelic ring $\A$ 
\begin{equation*}
1 \to F(\A) \to G^\sc(\A) \stackrel{\pi_\A}{\longrightarrow} G(\A) \stackrel{\psi_\A}{\surj} \cok(\pi_\A) \subset \prod_\fp H^1(\fg_\fp,F_\fp(\hat{K}_\fp^s))
\end{equation*}
where $\fg_\fp := \mathrm{Gal}(\hat{K}_\fp^s/\hat{K}_\fp)$ -- see \cite[\S~8.2]{PR} and $3)$ 
in the proof of Thm. 3.2. in \cite{Tha} for the function field case.  
According to \cite[Proposition~8.8]{PR} one has  
\begin{equation*} 
h_\iy(G) = [\psi_\A(G(\A)) : \psi_\A(G(\A_\iy) G(K))]. 
\end{equation*}
Denote $G^0(\A_\iy) = G_\iy(\hat{K}_\iy) \times \prod_{\fp \neq \iy} \un{G}_\fp^0(\hat{\CO}_\fp)$.  
Define the finite set $S := \{ \fp \mid \fp \mbox{ ramified} \}$. 
If $S = \emptyset$ then $G^0(\A_\iy) = G(\A_\iy)$ (see Remark \ref{G = G0 in unramified extension}). 
Otherwise, by the Borel density theorem (e.g.\ in the guise of \cite[Thm.~2.4,~Prop.~2.8]{Caprace/Monod:2009}) 
$\CG(\CO_{\{\iy \cup S})$ is Zariski-dense in $\prod_{\fp \in S \backslash \{ \iy\}}\un{G}_\fp$. 
This implies the equality $G(\A_\iy)G(K) = G^0(\A_\iy)G(K)$, and so
\begin{equation} \label{G class number}
h_\iy(G) = [\psi_\A(G(\A)) : \psi_\A(G^0(\A_\iy) G(K))]. 
\end{equation}  
Since $F$ is central in $G^\sc$, it is embedded in $T^\sc$. 
The corresponding exact sequence of $K$-groups of multiplicative type
$$ 1 \to F \to T^\sc \stackrel{\pi}{\rightarrow} T \to 1 $$
induces by $\fg$-cohomology the exact sequences over $K$ (see Remark \ref{H1 trivial over Kp}):
\begin{align*} 
1 \to F(K) \to G^\sc(K) \stackrel{\pi}{\rightarrow} G(K) \stackrel{\dl_K}{\longrightarrow} H^1(\fg,F(K^s)) \to 1 \\ \nonumber
1 \to F(K) \to T^\sc(K) \stackrel{\pi}{\rightarrow} T(K) \stackrel{\dl_K}{\longrightarrow} H^1(\fg,F(K^s)) \to 1 
\end{align*}
showing that $\dl_K(G(K)) = \dl_K(T(K))$. 
At any $\fp$, as $\un{G^\sc}_\fp$ is connected, by Lemma \ref{H1 connected trivial1} 
and Remark \ref{H1 trivial over Kp} one has
\begin{align*} 
\cok[\un{G^\sc}_\fp(\hat{\CO}_\fp) \to \un{G}^0_\fp(\hat{\CO}_\fp)] &= \cok[\un{T^\sc}_\fp(\hat{\CO}_\fp) \to \un{T}^0_\fp(\hat{\CO}_\fp)] = H^1(\la \s_\fp \ra,\un{F}_\fp(\hat{\CO}_\fp^\sh)), \\ \nonumber
\cok[G^\sc_\fp(\hat{K}_\fp) \to G_\fp(\hat{K}_\fp)] &= \cok[T^\sc_\fp(\hat{K}_\fp) \to T_\fp(\hat{K}_\fp)] = H^1(\fg_\fp,F_\fp(\hat{K}_\fp^s)).  
\end{align*}
Thus together with 
$[\un{G}_\fp(\hat{\CO}_\fp):\un{G}^0_\fp(\hat{\CO}_\fp)]= [\un{T}_\fp(\hat{\CO}_\fp):\un{T}_\fp^0(\hat{\CO}_\fp)]$ (see Remark \ref{index of components}), we may infer that
$$ \psi_\A(G(\A)) = \cok[G^\sc(\A) \to G(\A)] = \cok[T^\sc(\A) \to T(\A)] = \psi_\A(T(\A)). $$
In particular, over $\A_\iy$, due to Corollary \ref{surjection of local universal cover} Galois cohomology yields an exact sequence 
\begin{equation*} 
1 \to F(\A_\iy) \to G^\sc(\A_\iy) \stackrel{\pi_\A}{\rightarrow} G^0_\iy(\A_\iy) \stackrel{\psi_\A}{\longrightarrow} H^1(\fg_\iy,F_\iy(\hat{K}_\iy^s)) \times 
\prod_{\fp \neq \iy} H^1(\la \s_\fp \ra,\un{F}_\fp(\hat{\CO}_\fp^\sh)) \to 1
\end{equation*}
and similarly for the tori, showing that $\psi_\A(G^0(\A_\iy) = \psi_\A(T^0(\A_\iy))$. 
These cokernel equalities enable us to express $h_\iy(G)$ as given in (\ref{G class number}) via $T$, namely
\begin{equation} \label{G class number via T}
h_\iy(G) = [\psi_\A(T(\A)) : \psi_\A(T^0(\A_\iy) T(K))]. 
\end{equation}  
Applying the Snake Lemma on its two middle rows, we get the exactness of the diagram
$$\xymatrix{
1 \ar[r] & F(\A) \ar[r] \ar@{=}[d] & (T^\sc)^1(\A) \ar[r]^{\pi_\A} \ar@{^{(}->}[d] & T^1(\A) \ar[r]^{\psi_\A} \ar@{^{(}->}[d] & \psi_\A(T^1(\A)) \ar[r] \ar@{^{(}->}[d] & 1 \\ 
1 \ar[r] & F(\A) \ar[r] \ar[d] & T^\sc(\A) \ar[r]^{\pi_\A} \ar[d]^{\deg_{T^\sc}} & T(\A) \ar[r]^{\psi_\A} \ar[d]^{\deg_T} 
         & \psi_\A(T(\A)) \ar[r] \ar[d] & 1 \\
         & 0     \ar[r]  & \Hom(X^*(T^\sc)^\fg,\Z) \ar[r]^{\pi^\vee_K} \ar@{->>}[d] &\Hom(X^*(T)^\fg,\Z) \ar[r] \ar@{->>}[d] & \cok(\pi^\vee_K) \ar[r] \ar@{->>}[d] & 0 \\
         &               & \cok(\deg_{T^\sc})      \ar[r]^{\ov{\pi}^\vee_K}  &\cok(\deg_T)  \ar[r]  & \D \ar[r] & 0
}$$
(note that the elements in $\ker(\pi_\A)$ are units, and so belong to $T^1(\A)$)  
from which we see that: 
\begin{equation} \label{cok wh pi K}
[\psi_\A(T(\A)):\psi_\A(T^1(\A))] = |\cok(\pi^\vee_K)|/|\D|. 
\end{equation} 

Furthermore, from the following exact diagram
$$\xymatrix{
1 \ar[r] & F(\A) \ar[r] \ar@{=}[d] & (T^\sc)^c(\A)T^\sc(K) \ar[r]^{\pi_\A} \ar@{^{(}->}[d] & (T^c)^0(\A)T(K) \ar[r]^{\psi_\A} \ar@{^{(}->}[d] & \psi_\A((T^c)^0(\A)T(K)) \ar[r] \ar@{^{(}->}[d] & 1 \\ 
1 \ar[r] & F(\A) \ar[r] \ar[d] & (T^\sc)^1(\A) \ar[r]^{\pi_\A} \ar[d] & T^1(\A) \ar[r]^{\psi_\A} \ar[d] &\psi_\A(T^1(\A)) \ar[r] \ar[d] & 1 \\
         & 1     \ar[r] & \text{Cl}(T^\sc)  \ar[r]  & \text{Cl}(T^0) \ar[r] & \psi_\A(T^1(\A))/\psi_\A((T^c)^0(\A)T(K))  \ar[r]  & 1
}$$
with $(T^c)^0(\A) := \prod_\fp \un{T}_\fp^0(\hat{\CO}_\fp)$ one can see that
\begin{equation} \label{tori class numbers ratio}
\fc{h(T)}{h(T^\sc)} = \fc{h(T^0)/h(T^\sc)}{[T^c(\A)T(K) : (T^c)^0(\A)T(K)]} = \fc{[\psi_\A(T^1(\A)):\psi_A((T^c)^0(\A)T(K))]}{[T^c(\A)T(K) : (T^c)^0(\A)T(K)]}.  
\end{equation}
Using the third and second isomorphism Theorems one has 
\begin{align*}
T^c(\A)T(K) \Big/ (T^c)^0(\A)T(K) \cong T^c(\A)T(K)/T(K) \Big/ (T^c)^0(\A)T(K)/T(K) 
\cong T^c(\A)/\T(\BF_q) \Big/ (T^c)^0(\A)/\T^0(\BF_q) 
\end{align*}
whence
\begin{equation} \label{hT via hT0}
[T^c(\A)T(K) : (T^c)^0(\A)T(K)] = \fc{\prod_\fp [\un{T}_\fp(\hat{\CO}_\fp):\un{T}^0_\fp(\hat{\CO}_\fp)]}{[\T(\BF_q):\T^0(\BF_q)]}.
\end{equation}
Similarly, 
\begin{align} \label{TAiy via TcA}
       T^0(\A_\iy)T(K) \Big/ (T^c)^0(\A)T(K) &\cong T^0(\A_\iy)T(K)/T(K) \Big/ (T^c)^0(\A)T(K)/T(K) \\ \nonumber
&\cong T^0(\A_\iy)/\T^0(\BF_q) \Big/ (T^c)^0/\T^0(\BF_q) \cong T^0(\A_\iy)/ (T^c)^0(\A). 
\end{align}
In order to compute the cardinality of the latter ratio image under $\psi$, we may use cohomology again. 
Fix a separable closure $\hat{K}_\iy^s$ of $\hat{K}_\iy$ containing the maximal unramified extension $\hat{K}_\iy^\ur$ of $\hat{K}_\iy$ 
with absolute Galois group $\fg_\iy$ and inertia subgroup $I_\iy = \mathrm{Gal}(\hat{K}_\iy^s/\hat{K}_\iy^\ur)$. 
The spectral sequence then induces the exact sequence (see \cite[I.2.6(b)]{Ser3})
\begin{align*}
0 \to H^1(\la \s_\iy \ra,F_\iy(\hat{K}_\iy^\ur)) \stackrel{\text{inf}}{\rightarrow} H^1(\fg_\iy,F_\iy(\hat{K}_\iy^s)) \stackrel{\text{res}}{\rightarrow} H^1(I_\iy,F_\iy(\hat{K}_\iy^s))^{\s_\iy} \to H^2(\la \s_\iy \ra,F_\iy(\hat{K}_\iy^\ur)) = 0 
\end{align*}
which shows that 
\begin{equation} \label{dl T0(Aiy):dl T0(Ac)} 
[\psi_\A(T^0(\A_\iy)) : \psi_\A((T^c)^0(\A))] = \fc{|H^1(\fg_\iy,F_\iy(\hat{K}_\iy^s))|}{|H^1(\la \s_\iy \ra,\un{F}_\iy(\hat{\CO}_\iy^\sh))|} =  |H^1(I_\iy,F_\iy(\hat{K}_\iy^s))^{\s_\iy}| 
                                                   \stackrel{(\ref{what is k})}{=} t_\iy(G)
\end{equation}
All together we finally get 
\begin{align*}
h_\iy(G) \cdot \fc{h(T^\sc)}{h(T)} 
 &\stackrel{(\ref{G class number via T}),(\ref{tori class numbers ratio}),(\ref{hT via hT0})}
                   {=}                \fc{[\psi_\A(T(\A)):\psi_\A(T^0(\A_\iy)T(K))]} {[\psi_\A(T^1(\A)):\psi_\A((T^c)^0(\A)T(K))]} \cdot 
                                      \fc{\prod_\fp [\un{T}_\fp(\hat{\CO}_\fp):\un{T}^0_\fp(\hat{\CO}_\fp)]}{[\T(\BF_q):\T^0(\BF_q)]}   \\ \nonumber  
                   &=                 \fc{[\psi_\A(T(\A)):\psi_\A(T^1(\A))]} {[\psi_\A(T^0(\A_\iy)T(K)):\psi_\A((T^c)^0(\A)T(K))]}  \cdot 
                                      \fc{\prod_\fp [\un{T}_\fp(\hat{\CO}_\fp):\un{T}^0_\fp(\hat{\CO}_\fp)]}{[\T(\BF_q):\T^0(\BF_q)]} \\ \nonumber
 &\stackrel{(\ref{TAiy via TcA})}{=}  \fc{[\psi_\A(T(\A)):\psi_\A(T^1(\A))]} {[\psi_\A(T^0(\A_\iy)):\psi_\A((T^c)^0(\A))]}  \cdot 
                                      \fc{\prod_\fp [\un{T}_\fp(\hat{\CO}_\fp):\un{T}^0_\fp(\hat{\CO}_\fp)]}{[\T(\BF_q):\T^0(\BF_q)]} \\ \nonumber
 &\stackrel{(\ref{cok wh pi K}),(\ref{dl T0(Aiy):dl T0(Ac)})} {=} \fc{|\cok(\pi^\vee_K)|}{|\D| \cdot t_\iy(G)} \cdot 
                                      \fc{\prod_\fp [\un{T}_\fp(\hat{\CO}_\fp):\un{T}^0_\fp(\hat{\CO}_\fp)]}{[\T(\BF_q):\T^0(\BF_q)]}. \qedhere
\end{align*} 
\end{proof}

The following proposition now is an immediate consequence of the Main Theorem, Lemma \ref{class number ratio}.

\begin{prop} \label{D=1}
$|\D|=|\cok(\ov{\pi}^\vee_K)=1$. 
\end{prop}

\begin{proof}
Following
\cite{Ono3} 
by the proof of Theorem 6.1 in \cite{Behrend/Dhillon:2009} one has
\begin{equation} \label{BD}
\tau(G) = \tau(G^\sc) \cdot \fc{\tau(T)}{\tau(T^\sc)} \cdot |\cok(\wh{\pi}_K)|.  
\end{equation} 
Applying the functor $\Hom(-,\Z)$ on the sequence:  
\begin{equation} \label{global character groups}
0 \to X^*(T)^\fg \stackrel{\wh{\pi}_K}{\longrightarrow} X^*(T^\sc)^\fg \to M := \cok(\wh{\pi}_K) \to 0   
\end{equation} 
gives rise to the exact sequence 
$$ 
0 \to 0 = \Hom(M,\Z) \to \Hom(X^*(T^\sc)^\fg,\Z) \stackrel{\pi^\vee_K}{\longrightarrow} \Hom(X^*(T)^\fg,\Z) \to \Ext^1(M,\Z) \cong \Hom(M,\BQ/\Z) \to 0 
$$ 
which shows that $\cok(\pi^\vee_K)$ is the Pontryagin dual of $\cok(\wh{\pi}_K)$. 
As both groups are finite, they therefore have the same cardinality.  
Hence from formula (\ref{BD}) we get 
\begin{eqnarray*}
\tau(G) & = 
  & \tau(G^\sc) \cdot |\cok(\pi^\vee_K)| \cdot \fc{\tau(T)}{\tau(T^\sc)}  \\ \nonumber
  & \stackrel{(\ref{Tamagawa isogeny})}{=} 
  & \tau(G^\sc) \cdot |\cok(\pi^\vee_K)| \cdot \fc{h(T)}{h(T^\sc)} \cdot \fc{w(T^\sc)}{w(T)} 
    \prod_\fp L_\fp(1,\chi_{T_\fp}) \cdot \om_\fp(\un{T}_\fp(\hat{\CO}_\fp)) \\   \nonumber
  & \stackrel{\mbox{\cite[3.2]{Bit}}}{=} 
  & \tau(G^\sc) \cdot |\cok(\pi^\vee_K)| \cdot \fc{h(T)}{h(T^\sc)} \cdot \fc{w(T^\sc)}{w(T)} 
    \prod_\fp [\un{T}_\fp(\hat{\CO}_\fp):\un{T}_\fp^0(\hat{\CO}_\fp)]  \\  \nonumber
  & \stackrel{\ref{wratio}}{=} 
  & \tau(G^\sc) \cdot |\cok(\pi^\vee_K)| \cdot \fc{h(T)}{h(T^\sc)} \cdot 
    \fc{\prod_\fp [\un{T}_\fp(\hat{\CO}_\fp):\un{T}_\fp^0(\hat{\CO}_\fp)]}{[\T(\BF_q):\T^0(\BF_q)]}  \\  \nonumber
  & \stackrel{\ref{class number ratio}}{=}  
  & \tau(G^\sc) \cdot h_\iy(G) \cdot t_\iy(G) \cdot |\D| \\ \nonumber
  & \stackrel{\ref{maintheorem}}{=}  
  & \tau(G^\sc) \cdot \tau(G) \cdot |\D| . 
\end{eqnarray*} 
This implies $|D| = \frac{1}{\tau(G^\sc)}=1$ due to the Weil conjecture. 
\end{proof}

\begin{remark}
Any isogenous $K$-tori $T^\sc$ and $T$ with $T^\sc$ quasi-trivial 
can be realized as Cartan subgroups of semisimple and quasi-split groups $G^\sc$ and $G$ respectively, 
with $G^\sc$ simply connected. 
E.g., given the isogeny $\pi:T^\sc \to T$, then each factor $R_{L/K}(\BG_m^d)$ in $T^\sc$, 
is a Cartan subgroup of the quasi-split and simply connected group $G^\sc = R_{L/K}(\textbf{SL}_{d+1})$,   
and $T$ is a Cartan subgroup of $G = G^\sc/\ker(\pi)$, cf. Examples \ref{quadratic non split} -- \ref{generalized example} below. 
Hence we may generalize Proposition~\ref{D=1} to the statement that for any isogeny $T^\sc \to T$, 
the induced map $\cok(\deg_{T^\sc}) \to \cok(\deg_T)$ is surjective.  
\end{remark}

Quite naturally, our Main Theorem reproduces the following well-known facts. 
Recall that in this section we assume the validity of the Weil Conjecture.

\begin{cor} \label{G split}
If $G$ is $K$-split and $g=0$ then $h_\iy(G)=1$ and $\tau(G) = t_\iy(G) = |F|$. 
\end{cor}

\begin{proof}
If $G$ is $K$-split, then $T^\sc$ and $T$ are $K$-split thus having connected reduction everywhere and $h(T) = h(T^\sc)$. 
Furthermore, $|\cok(\pi^\vee_K)| = |F| = |F_\iy| = t_\iy(G)$ whence by Lemma \ref{class number ratio} $h_\iy(G) = 1$.  
Hence according to the Main Theorem \ref{maintheorem} we get $\tau(G) = t_\iy(G) = |F_\iy| = |F|$. 
\end{proof}

\begin{remark} \label{PGLn over elliptic curve}
We have assumed in this section that the genus $g$ of $C$ is $0$. 
Otherwise, if $g>0$, $h_\iy(G)$ does not need to be $1$, though $G$ splits over $K$.  
For example, let $G=\textbf{PGL}_n$ defined over $K = \BF_q(C)$ where $C$ is an elliptic curve ($g=1$). 
Let $\iy$ be a $K$-rational point and let $\CG$ be an affine, 
smooth, flat, connected and of finite type model of $G$ defined over $\Sp \Oiy$ 
as been constructed above. 
Let $\CGL_n$ be a similar construction for $\textbf{GL}_n$ and $\CG_m$ for $\BG_m$.  
According to Nisnevich exact sequence (see \cite[3.5.2]{Nis} and also \cite[Thm.~3.4]{Gon}),  
since the Shafarevich-Tate group w.r.t. $S=\{\iy\}$ is trivial in this split case, we have: 
\begin{equation*}
\mathrm{Cl}_\iy(G) \cong H^1_\et(\Oiy,\CG). 
\end{equation*}
The exact sequence of smooth $\Oiy$-groups
\begin{equation*}
1 \to \CG_m \to \CGL_n \to \CG \to 1, 
\end{equation*}
gives rise by flat cohomology to the following exact sequence 
\begin{equation*}
\Pic(C^\af) \stackrel{\partial}{\to} H^1_\et(\Oiy,\CGL_n) \stackrel{\delta}{\to} H^1_\et(\Oiy,\CG) \to H^2_\et(\Oiy,\BG_m) 
\end{equation*} 
on which $H^1_\et(\Oiy,\CGL_n)$ classifies the rank $n$ vector bundles defined over $C^\af := C - \{\iy\}$. 
Every rank-$n$ vector bundle over a Dedekind domain is a direct sum $\CO_{C^\af}^{n-1} \oplus \CL$, 
where $[\CL] \in \Pic(C^\af)$ and $\CO_{C^\af}$ is the trivial line bundle.  
As $\partial: [\CL] \mapsto n \CL$ we have: 
$$ 
\im(\delta) \cong H^1_\et(\Oiy,\CGL_n)/\ker(\delta) = H^1_\et(\Oiy,\CGL_n)/\im(\partial) \cong \Pic(C^\af)[n] := \Pic(C^\af) / n\Pic(C^\af). 
$$  
Moreover, as $C^\af$ is smooth, one has (see \cite[Prop.~2.15]{Mil1}):
$H^2_\et(\Oiy,\CG_m) = \mathrm{Br}(\Oiy)$,  
classifying Azumaya $\Oiy$-algebras (see \cite[\S~2]{Mil1}). 
At each prime $\fp$: $\mathrm{Br}(\Oiy) \subset \mathrm{Br}((\Oiy)_\fp) = \mathrm{Br}(\CO_\fp)$.  
Since $\CO_\fp$ is complete, the latter group is isomorphic to $\mathrm{Br}(k_\fp)$ 
(see \cite[Thm.~6.5]{AG:1960}).  
But $k_\fp$ is a finite field thus $\mathrm{Br}(k_\fp)$ is trivial as well as $H^2_\et(\Oiy,\BG_m)$ 
and $\delta$ is surjective.     
We get that $h_\iy(G) = |\mathrm{Cl}_\iy(G)| = |H^1_\et(\Oiy,\CG)| = |\Pic(C^\af)[n]|$. 
In order to compute this group, as we assumed $\iy$ is $K$-rational 
and $(\text{char}(K),|F|)=1$,  
the restrictions of $C$ to $C^\af$ and of $\Pic(C)$ for $\Pic^0(C)$ inducing the exact sequences 
\begin{align*}
0 \to \Z \to \Pic(C) &\to \Pic(C^\af) \to 0 \\ \nonumber
0 \to \Z \to \Pic(C) &\to \Pic^0(C) \to 0 
\end{align*}
give an isomorphism $\Pic(C^\af) \cong \Pic^0(C) \cong C(\BF_q)$.  
So it is easy to find an elliptic curve $C$ for which $h_\iy(G) = |C(\BF_q)[n]| > 1$. \\
Applying flat cohomolopgy on Kummer's exact sequence of $\Oiy$-schemes:
$$ 1 \to \mu_n \to \CG_m \stackrel{x \mapsto x^n}{\longrightarrow} \CG_m $$
gives rise to the exact sequence of groups of $\Oiy$-points:
$$ 
(\Oiy)^\times \stackrel{x \mapsto x^n}{\longrightarrow} (\Oiy)^\times 
                  \to H_\et^1(\Oiy,\mu_n) \to \Pic(C^\af) \stackrel{z \mapsto nz}{\longrightarrow} \Pic(C^\af) 
$$
which in light of the proof of Lemma \ref{Fiscoh2} can we rewritten as
\begin{equation*} 
1 \to \BF_q^\times / (\BF_q^\times)^n \to H_\et^1(\Oiy,\mu_n) \to \Pic(C^\af)[n] \to 0.   
\end{equation*}
We deduce that $H^1_\et(\Oiy,\mu_n)$ is an extension of $\BF_q^\times/(\BF_q^\times)^n$ 
by $\Pic(C^\af)[n]$ and so  
$$ 
|H^1_\et(\Oiy,\mu_n)| = |\BF_q^\times/(\BF_q^\times)^n| \cdot |\Pic(C^\af)[n]| = |H^1(\BF_q,\mu_n)| \cdot |\Pic(C^\af)[n]|. 
$$
Consequently 
$$ 
j_\iy(G) = \fc{|H^1_\et(\Oiy,\mu_n)|}{|\mu_n(\BF_q)|} = \fc{|H^1_\et(\Oiy,\mu_n)|}{|H^1(\BF_q,\mu_n)|} = |\Pic(C^\af)[n]| = h_\iy(G)
$$ 
and finally:  
$$ \tau(G) = h_\iy(G) \cdot \fc{t_\iy(G)}{j_\iy(G)} = |\wh{F_\iy}^{\fg_\iy}| = |F| = n. $$
\end{remark}

\begin{cor} \label{G adjoint}
If $G$ is adjoint (not necessarily split) and $g=0$ then $h_\iy(G)=1$ and $\tau(G) = t_\iy(G) = |\wh{F}^\fg|$, 
where $\wh{F}:=\mathrm{Hom}(F(K^s),\BG_{m,K^s})$ and $\fg := \mathrm{Gal}(K^s/K)$.   
\end{cor}

\begin{proof}
According to Ono's formula $(3.9.11')$ in \cite{Ono1}, 
considering the isogeny of class groups of $T^\sc$ and $T$, 
there exists a finite set of primes $S$ for which 
$$ \fc{h(T)}{h(T^\sc)} =  \left(\fc{q(\a_S^1)}{\prod_{\fp \in S} q(\a_{\CO_\fp})} \right) 
                   \Big / \left(\fc{q(\a_K^S)}{q(\a_W)} \right) $$
where for any isogeny $\a$, $q(\a)$ stands for $|\cok(\a)| / |\ker(\a)|$ and (see notation in Section \ref{section4}): 
$$  T_S^1(\A) := T^1(\A) \cap T_S,              \ \ T_S(K) := T(\A(S)) \cap T(K),  $$
$$     \a_S^1 := (T^\sc_S)^1(\A) \to T^1_S(\A), \ \ \a_K^S := T^\sc_S(K) \to T_S(K), \ \ \a_W:= W(T^\sc) \to W(T). $$  
As $G^\sc$ is simply-connected and $G$ is adjoint, both quasi-split, 
their Cartan subgroups $T^\sc$ and $T$ are quasi-trivial 
and their integral models are connected everywhere (see Remark \ref{quasi-trivial Cartan}). 
In this case the quantities $q(\a)$ related to $\a^S$ and $\a_K^S$ 
are equals to the ones obtained in the split case on which the class group of each $\BG_m$ 
is the class group of $K$ (see Formulas (3.1.7) and (3.1.8) in \cite{Ono1}), 
thus equal to $1$. 
Also by Lemma \ref{volumes of compact groups} one may deduce 
that $q(\a_{\CO_\fp})=1$ at each $\fp$ and by Lemma \ref{wratio} (recall $g=0$) this can be deduced also for $\a_W$.   
Hence $T^\sc$ and $T$ share the same class number and so by Lemma \ref{class number ratio}, Prop. \ref{D=1} and our Main Theorem 
$\tau(G) = h_\iy(G) \cdot t_\iy(G) = |\cok(\pi^\vee_K)|$ (see in the proof of \ref{D=1}). 
But as $T$ is quasi-trivial, $X^*(T) = \bigoplus_{i=1}^n \mathrm{Ind}_{\{id\}}^{H_i}(\BZ)$ 
where $H_i$ are some finite subgroups of $\fg$,  
thus by Shapiro's lemma $H^1(\fg,X^*(T)) \cong \bigoplus H^1(H_i,\Z)=0$.    
Consequently the exact sequence of character groups (considered as $\fg$-modules): 
$$ 0 \to X^*(T) \to X^*(T^\sc) \to \wh{F} \to 0 $$
gives rise by $\fg$-cohomology to the exact sequence: 
$$ 0 \to X^*(T)^\fg \stackrel{\wh{\pi}_K}{\longrightarrow} X^*(T^\sc)^\fg \to \wh{F}^\fg \to H^1(\fg,X^*(T)) = 0 $$
from which we can see that $\tau(G) = |\cok(\pi^\vee_K)| = |\cok(\wh{\pi}_K)| = |\wh{F}^\fg|$. 
This also shows by Ono's formula \cite[Main Theorem]{Ono3} (see Cor. \ref{onoformula} below) that $\Sh^1(\wh{F})=1$.    
\end{proof}

More generally, our Main Theorem leads us  
under this section settings: $G$ is quasi-split and $g=0$,  
to the following more general result obtained by Ono at $1965$ 
(see Main Theorem in \cite{Ono3}). 
It was designed for groups over number fields 
and been generalized by Behrend and Dhillon at $2009$ 
to the function field case in \cite[Theorem~6.1]{Behrend/Dhillon:2009}.

\begin{cor}[Ono's formula] \label{onoformula} 
One has
$$ \tau(G) = \fc{|\wh{F}^\fg|}{|\Sh^1(\wh{F})|} $$ 
where the denominator is the first Shafarevitch--Tate group assigned to $\wh{F}$. 
\end{cor}

\begin{proof}
Applying Galois $\fg$-cohomology to the sequence of groups of characters
$$ 0 \to X^*(T) \stackrel{\wh{\pi}}{\rightarrow} X^*(T^\sc) \to \wh{F} \to 0 $$
where $\wh{F}:=\Hom(F \otimes_K K^s,\BG_{m,K^s})$ yields the relation 
\begin{equation} \label{cokKcard}
|\cok(\wh{\pi}_K)| = \fc{|\wh{F}^\fg|}{|H^1(\fg,X^*(T))|}. 
\end{equation} 
The following formula for the Tamagawa number of a torus is taken from \cite[Main~Theorem]{Ono2}, 
\cite[Corollary~3.3]{Oes} 
\begin{equation} \label{Onotorus}
\tau(T) = \fc{|H^1(\fg,X^*(T))|}{|\Sh^1(T)|}. 
\end{equation}
Together with $|\Sh^1(T)| = |\Sh^1(\wh{F})|$  (\cite[p.~102]{Ono3}) 
we conclude
\begin{align} \label{tauG tauT}
\tau(G) & \stackrel{(\ref{BD})}{=}        \tau(G^\sc) \cdot \tau(T) \cdot |\cok(\wh{\pi}_K)| \notag \\
        & \stackrel{(\ref{cokKcard})}{=}  \tau(G^\sc) \cdot \tau(T) \cdot \fc{|\wh{F}^\fg|}{|H^1(\fg,X^*(T))|} \notag \\
        & \stackrel{(\ref{Onotorus})}{=}  \tau(G^\sc) \cdot \fc{|\wh{F}^\fg|}{|\Sh^1(\wh{F})|}. \qedhere
\end{align}
\end{proof}

\bk

In the following examples we refer to a construction which was demonstrated by Ono over number fields, in \cite[Example~6.3]{Ono2}. 
Our ground field is $K=\BF_q(t)$ with odd characteristic and $\iy$ is chosen to correspond to the pole of $t$. 
At each example we consider another extension $L$ of $K$.  
We denote $\fg = \mathrm{Gal}(L/K)$.  
The group $G^\sc = R_{L/K}(\textbf{SL}_2)$ is the universal cover 
of the semisimple and quasi-split $K$-group $G = G^\sc / F$ 
where $F:=R^{(1)}_{L/K}(\mu_n) = \ker[R_{L/K}(\mu_n) \to \mu_n]$.  
Let $S$ be the diagonal $K$-split maximal torus in $G$.   
Then $T = \text{Cent}_G(S)$ is a maximal torus of $G$ 
and is isomorphic as a $\fg$-module to the $K$-torus $\BG_m \times R^{(1)}_{L/K}(\BG_m)$ 
where the right hand factor is the norm torus, namely the kernel of the norm map (see \cite[Example~5.6]{San})
$$  R_{L/K}^{(1)}(\BG_m) := \ker \left[R_{L/K}(\BG_m) \stackrel{N_{L/K}}{\longrightarrow} \BG_m \right]. $$
Its preimage in $G^\sc$ is the Weil torus $T^\sc = R_{L/K}(\BG_{m,L})$, 
fitting into the exact sequence
$$ 1 \to F \to T^\sc \stackrel{\pi}{\longrightarrow} T \to 1. $$
Over any $\hat{\CO}_\fp$, the norm torus is $\Sp \hat{\CO}_\iy[a,b]/(a^2-\fp b^2-1)$.   
Its reduction provides at each place $\fp$, $e_\fp$ connected components, 
where $e_\fp$ stands for the ramification index there (see \cite[Example~3.3]{Bit}),    
i.e. $[\un{T}_\fp(\hat{\CO}_\fp):\un{T}^0_\fp(\hat{\CO}_\fp)] = e_\fp$. 
In this construction $|\cok(\wh{\pi}_K)|=1$ .

\begin{example} \label{quadratic non split}
We start by $L = \BF_{q^2}(t)$ obtained by extending the field of constants of $K$. 
Since the extension is quadratic, $F_\iy = \mu_2$ is $\hat{K}_\iy$-split whence $t_\iy(G) = |\wh{F_\iy}| = |F_\iy| = 2$. 
Moreover, as $L/K$ is imaginary and totally unramified we have $h(T)/h(T^\sc) = 2$ (see \cite[Example~1]{Mor}). 
Thus by Lemma~\ref{class number ratio} $h_\iy(G) = 1$     
whence according to our Main Theorem $\tau(G) = h_\iy(G) \cdot t_\iy(G) = 2$.  
\end{example} 

\begin{example}
Now let $L=K(\sqrt{d})$ where $d$ is a product of $m$ distinct finite primes $\fp_i$. 
As before, $F_\iy = \mu_2$ and $t_\iy(G)=2$.  
Recall that the norm torus is the only factor in $T^\sc$ and $T$ which might have a disconnected reduction. 
This time, since each $\fp_i$, as well as $\iy$, ramifies in $L$ with $e_\fp = 2$ we have 
$$ \prod_\fp[\un{T}_\fp(\hat{\CO}_\fp):\un{T}_\fp^0(\CO_{\fp})] = 2^{m+1} $$
while:
$$ [\T(\BF_q):\T^0(\BF_q)] = |\{x \in \BF_q : x^2=1 \}| = |\{\pm 1\}| = 2. $$
Moreover, as $h(T^\sc)/h(T) = 2^{m-1}$ (see \cite[Example~1]{Mor})  
and $\cok(\pi^\vee_K)=1$, by Lemma \ref{class number ratio} we get $h_\iy(G) =1$. 
Altogether, we see by the Main Theorem that $\tau(G)$ remains equal to $2$, independently of $m$. 
Both this result and the one of the previous example agree with Ono's formula \ref{onoformula}; 
indeed, as $L/K$ is cyclic, $\Sh^1(\wh{F})=1$ and $\tau(G) = |\wh{F}^\fg| = |F| = 2$. 
\end{example}

\begin{example} \label{generalized example}
Let $L = K(\Lm_f)$ be the $f$-cyclotomic extension where $f$ is an irreducible polynomial of degree $d$. 
Then $\fg$ is cyclic of order $n=q^d-1$. 
We still have
$h(T)/h(T^\sc) =1$ (\cite[Example~2]{Mor}). 
The only places which ramify in $L$ are $\iy$ with $e_\iy = q-1$ 
and $(f)$ which is totally ramified (see \cite[Theorem~3.2]{Hay}). 
Therefore $[\un{T}_{(\iy)}(\CO_{(\iy)}):\un{T}^0_{(\iy)}(\CO_{(\iy)})] = q-1$ and 
$[\un{T}_{(f)}(\CO_{(f)}):\un{T}^0_{(f)}(\CO_{(f)})] = n$.   
On the units group, since $q-1|n$ we have
$$ [\T(\BF_q):\T^0(\BF_q)] = |\{x \in \BF_q: x^n=1 \}| = |\BF_q^\times| = q-1. $$
Moreover, $t_\iy(G) = |\wh{F_\iy}^{\fg_\iy}| = |\mu_n| = n$ and as before $\cok(\wh{\pi}_K)=1$. 
So by Lemma \ref{class number ratio} we get 
$$ h_\iy(G) = \fc{\prod_\fp[\un{T}_\fp(\hat{\CO}_\fp):\un{T}^0_\fp(\hat{\CO}_\fp]}{t_\iy(G) \cdot [\T(\BF_q):\T^0(\BF_q)]} = \fc{(q-1) \cdot n}{n \cdot (q-1)} = 1. $$
Thus by our Main Theorem we conclude that $\tau(G) = t_\iy(G) = n$.   
Indeed, as $L/K$ is cyclic, $\Sh^1(\wh{F})=1$ and $\tau(G) = |\wh{F}^\fg| = |\mu_n| = n$, 
which agrees again with Ono's formula \ref{onoformula}.  
\end{example}  

\bk


\end{document}